\documentclass[11pt,twoside, reqno]{article}

\usepackage{amssymb}
\usepackage{amsmath}
\usepackage{mathrsfs}
\usepackage{amsthm}
\usepackage{amsfonts}
\usepackage{txfonts}
\usepackage{enumerate}
\usepackage{authblk}
\usepackage[retainorgcmds]{IEEEtrantools}

\usepackage{latexsym,amssymb}
\usepackage{color}

\usepackage{graphicx}

\usepackage{pgf,tikz}
\usepackage{mathrsfs}
\usetikzlibrary{arrows}
\definecolor{ttttff}{rgb}{0.2,0.2,1.}
\definecolor{ttffcc}{rgb}{0.2,1.,0.8}
\definecolor{qqqqff}{rgb}{0.,0.,1.}

\definecolor{zzttqq}{rgb}{0.6,0.2,0.}
\definecolor{qqqqff}{rgb}{0.,0.,1.}

\allowdisplaybreaks

\usepackage{indentfirst}

\newtheorem{theorem}{Theorem}[section]
\newtheorem{lemma}[theorem]{Lemma}
\newtheorem{corollary}[theorem]{Corollary}
\newtheorem{proposition}[theorem]{Proposition}
\theoremstyle{definition} \newtheorem{remark}[theorem]{Remark}

\newtheorem{example}[theorem]{Example}

\newcommand{\scr}[1]{\mathscr #1} \numberwithin{equation}{section}

\newcommand\E{\mathbb{E}} \newcommand\R{\mathbb{R}}

 \newcommand\sP{\mathscr P} \newcommand\sC{\mathscr C}

\def\tr{{\operatorname{tr}}\kern 0.08em} 

\def\Ent{{\operatorname{Ent}}} \def\Ric{{\operatorname{Ric}}}
 
\def\Hess{{\operatorname{Hess}}} \newcommand{\ptr}{/\!/}
\newcommand\1{\hbox{\kern.375em\vrule
    height1.57ex depth-.1ex width.05em\kern-.375em \rm 1}}

\def\vol{\mathpal{vol}} 

\newcommand\HS{{\text{\rm\tiny HS}}}
\def\mathpal#1{\mathop{\mathchoice{\text{\rm #1}}%
    {\text{\rm #1}}{\text{\rm #1}}%
    {\text{\rm #1}}}\nolimits}

\def\vd{\mathrm{d}} \def\r{\right} \def\l{\left}
\def\e{\operatorname{e}} 

\def\mathpal#1{\mathop{\mathchoice{\text{\rm #1}}%
    {\text{\rm #1}}{\text{\rm #1}}%
    {\text{\rm #1}}}\nolimits} \def\id{{\mathpal{id}}}

\def\ff{\frac}  \def\nn{\nabla} 
\def\<{\langle} \def\>{\rangle}  
 \def\1{\mathds{1}}

    \def\rr{ \right|\right|}

\def\ff{\frac} \def\ss{\sqrt}  
  
 \def\DD{\Delta}  \def\rr{\rho}
\def\<{\langle} \def\>{\rangle} \def\GG{\Gamma} \def\gg{\gamma}
\def\nn{\nabla}  \def\E{\mathbb E}
\def\d{\text{\rm{d}}} \def\bb{\beta}  
  \def\beg{\begin}
  \def\beq{\begin{equation}} 
    \def\e{\text{\rm{e}}}

     \def\C{\scr C}

    \def\to{\rightarrow} \def\iint{\int}
    \def\W{\mathbb W}
    \def\S{\mathbb S}
    \def\Ent{{\rm Ent}} 
    
      \def\l{\left}\def\r{\right}
    \newcommand{\eps}{\varepsilon}
    \newcommand\mequal{\overset{\text{\scriptsize\rm m}}{=}}

    \usepackage{txfonts}
    \AtBeginDocument{\mathcode`v=\varv}
    \AtBeginDocument{\mathcode`w=\varw}
    \makeatother

\begin{document}

    \arraycolsep=1pt

    \title{\bf\Large Some inequalities on
      Riemannian manifolds linking Entropy,\\ Fisher
      information, Stein discrepancy and Wasserstein distance
      \footnotetext{\hspace{-0.35cm} 2010 {\it
          Mathematics Subject Classification}. Primary: 60E15;
        Secondary: 35K08, 46E35, 42B35.\endgraf {\it Key words and
          phrases}. Relative entropy, Fisher information, Stein
        discrepancy, Wasserstein distance.
      }}

    \author[1]{Li-Juan Cheng} \author[2]{Anton Thalmaier}
    \author[3,4]{Feng-Yu Wang}
    \setlength{\affilsep}{3em}\renewcommand*{\Authsep}{,
    }\renewcommand*{\Authand}{, }\renewcommand*{\Authands}{, }

    \affil[1]{\small School of  Mathematics, Hangzhou Normal
      University,\par
      Hangzhou 311121, People's Republic of China\par
      \texttt{lijuan.cheng@hznu.edu.cn}\vspace{1em}}

    \affil[2]{\small Department of Mathematics, University of
      Luxembourg, Maison du Nombre,\par
      L-4364 Esch-sur-Alzette, Luxembourg\par
      \texttt{anton.thalmaier@uni.lu}\vspace{1em}}

    \affil[3]{\small Center for Applied Mathematics, Tianjin
      University,\par Tianjin 300072, People's Republic of China}
    \affil[4]{\small Department of Mathematics, Swansea University,
      Bay Campus,\par
      Swansea SA1 8EN, United Kingdom\par
      \texttt{wangfy@tju.edu.cn}}

    \date{\today}
    \maketitle

    \begin{abstract}\noindent
      For a complete connected Riemannian manifold $M$ let
      $V\in C^2(M)$ be such that $\mu(\vd x)=\e^{-V(x)} \vol(\vd x)$
      is a probability measure on $M$.  Taking $\mu$ as reference
      measure, we derive inequalities for probability measures on~$M$
      linking relative entropy, Fisher information, Stein discrepancy
      and Wasserstein distance. These inequalities strengthen in
      particular the famous log-Sobolev and transportation-cost
      inequality and extend the so-called
      Entropy/Stein-discrepancy/Information (HSI) inequality 
      established by Ledoux, Nourdin and Peccati (2015) for the standard
      Gaussian measure on Euclidean space to the setting of Riemannian
      manifolds.
    \end{abstract}

    \section{Introduction}

    Let $ \gamma(\vd x)=(2\pi)^{-n/2}\e^{-|x|^2/2}\,\vd x$ be the
    standard Gaussian measure on $\R^n$ and denote by $\sP(\R^n)$ the
    set of probability measures on $\R^n$. The classical log-Sobolev
    inequality \cite{Gross} indicates that \begin{equation} \label{LS}
      H(\nu\,|\,\gamma)\le \frac12 I(\nu\,|\,\gamma),\ \ \nu\in
      \sP(\R^n),
    \end{equation} and the transportation-cost inequality \cite{Talagrand} states that
    \begin{equation}\label{TSP} \W_2(\nu,\gamma)^2\le 2
      H(\nu\,|\,\gamma),\quad
      \nu\in \scr P(\R^n),\end{equation}
where for $\nu,\mu\in \scr P(\R^n)$ we consider
    \begin{enumerate}[1.]
    \item the relative entropy of $\nu$ with respect to $\mu$,\beq\label{D1} H(\nu\,|\,\mu):=\beg{cases} \displaystyle\int_{\R^n}h\log h\,\d\mu, &\text{if}\ \nu(\d x)= h(x)\mu(\d x),\\
      \infty, \ &\text{otherwise,}\end{cases} \end{equation}
\item the Fisher information of $\nu$ with respect to $\mu$
\begin{equation}\label{D2}
  I(\nu\,|\,\mu):=\beg{cases} \displaystyle\int_{\R^n}\frac{|\nabla h|^2}{h}\, \vd \mu,\ &\text{if}\ \nu(\d x)= h(x)\mu(\d x), \ss h\in W^{1,2}(\mu),\\
  \infty,\ &\text{otherwise,}\end{cases}
\end{equation}
\item the $L^2$-Wasserstein distance $\W_2$ of $\mu$ and
$\nu$, i.e.  \beq\label{D3} \W_2(\mu,\nu):=\inf_{\pi\in \C(\mu,\nu)}
\l(\int_{\R^n\times\R^n} |x-y|^2\,\pi(\d x,\d
y)\r)^{1/2} \end{equation} with
$\C(\mu,\nu)$ being the set of all couplings of $\mu$ and $\nu$.
\end{enumerate}

Inspired by \cite{NPS14}, Ledoux, Nourdin and Peccati \cite{LNP15}
established some new type of inequalities improving \eqref{LS} and
\eqref{TSP} by adopting the Stein discrepancy
$S(\nu\,|\,\gamma)$ of $\nu$ with respect to~$\gg$ as further ingredient.
This quantity is defined as
\begin{equation}\label{D4}
  S(\nu\,|\,\gamma):=\inf_{\tau\in \S_\nu}\l(\int_{\R^n}|\tau-\id|_\HS^2\, \vd\nu\r)^{1/2}
\end{equation}
where $\id$ is the $n\times n$-identity matrix and
$\S_\nu$ the set of measurable maps
$\tau\in L_{\rm loc}^1( \R^n\to \R^n\otimes\R^n;\nu)$ such that
\begin{align*}
  \int_{\R^n} x\cdot \nabla \varphi\, \vd \nu=\int_{\R^n}\langle \tau, \Hess_{\varphi}\rangle_\HS\, \vd\nu,\quad \varphi\in C_0^\infty(\R^n).
\end{align*}
A map $\tau\in \S_\nu$ is called a Stein kernel of $\nu$. In general, the set
$\S_\nu$ may contain infinitely many maps; for instance, for the
Gaussian measure $\gg$,
$$\left\{x\mapsto \big(1+r \e^{|x|^2/2}\big)\, \id\colon r\in\R\right\}\subset\S_\gg.$$
Recall that however the Gaussian measure
$\gamma$ is characterized as the only probability distribution on
$\R^n$ satisfying
$$
\int_{\R^n} x\cdot \nabla \varphi\, \vd \gamma=\int_{\R^n}\Delta
\varphi\, \vd\gamma,\quad \varphi\in C_0^\infty(\R^n).
$$
Hence for $\nu\in \scr P(\R^n)$ it holds that $\id\in
\S_\nu$ if and only if $\nu=\gg$.

This equivalence indicates that the Stein discrepancy $S(\nu\,|\,\gamma)$
with respect to the Gaussian distribution $\gamma$ provides a natural
measure for the proximity of $\nu$ to $\gg$ and allows to quantify how
far $\nu$ is away from $\gg$.  It is a crucial quantity for normal
approximations and appears implicitly in many works on Stein's method
\cite{Stein}.  The Stein method was initially developed to quantify
the rate of convergence in the Central Limit Theorem
\cite{SteinMethod}, and has recently been extended to probability
distributions on Riemannian manifolds \cite{Th2020}.  For Gamma
approximations the Stein discrepancy represents the bound one
customarily obtains when applying Stein's method to measure the distance
to the one-dimensional Gamma distribution, see
\cite{Ch-G-Sh11,DP18,LNP15,N-P12}.

Recall that the relative entropy $H(\nu\,|\,\gamma)$ is another measure of the
proximity between $\nu$ and $\gamma$ (note that $H(\nu\,|\,\gamma)\geq0$ and
$H(\nu\,|\,\gamma)=0$  if and only if $\nu=\gamma$) which is moreover stronger
than the total variation distance, $2$TV$(\nu,\gamma)^2\leq H(\nu\,|\,\gamma)$,
see \cite{Villani, LNP15}.

Considering the Stein discrepancy
$S(\nu\,|\,\gamma)$ as a new ingredient,
according to \cite[Theorem 2.2]{LNP15}, one has the
following HSI inequality which strengthens \eqref{LS}:
\beq\label{HSI} H(\nu\,|\,\gamma)\leq
\frac12 S^2(\nu\,|\,\gamma)\log
\bigg(1+\frac{I(\nu\,|\,\gamma)}{S^2(\nu\,|\,\gamma)}\bigg),\quad \nu\in
\scr P(\R^n),
\end{equation}
whereas the inequality \cite[Theorem 3.2]{LNP15},
\begin{equation}\label{W2}
  \W_2(\nu\,|\,\gg)\leq  S(\nu\,|\,\gg) \arccos\left(\exp\l(-\frac{H(\nu\,|\,\gg)}{S^2(\nu\,|\,\gg)}\r)\right),\quad \nu\in \scr P(\R^n),
 \end{equation} improves the transportation-cost inequality \eqref{TSP}.
Moreover, \cite[Theorem 2.8]{LNP15} gives the existence of a constant $C>0$
such that
\begin{equation}\label{FP}
  \l(\int |f|^p\, \vd \nu\r)^{1/p}\leq C\left(S_p(\nu\,|\,\gg)+\sqrt{p}\,\Big(\int |\tau|_{\rm op}^{p/2}\vd \nu\Big)^{1/p}\right),\quad \nu(f)=0,\ |\nn f|\le 1,\ \tau\in \S_\nu,
\end{equation}
where for $p\geq1$, one defines
\begin{equation}\label{D4New}
  S_p(\nu\,|\,\gamma):=\inf_{\tau\in \S_\nu}\l(\int_{\R^n}|\tau-\id|_\HS^p\,\vd\nu\r)^{1/p}.
\end{equation}
In particular $S_2(\nu\,|\,\gamma)$ is the Stein discrepancy as defined above.

In \cite{LNP15} these inequalities have been extended to probability measures $\mu(\d x):=\e^{V(x)}\d x$ on $\R^n$ which are stationary distributions of an elliptic symmetric diffusion process on $\R^n$. The required
assumptions are formulated in terms of conditions on the iterated Bakry-\'Emery operators $\GG_i$ 
$(i=1,2,3)$. It is worth mentioning that the analysis towards the HSI bound in this context
makes crucial use of the
iterated gradient $\GG_3$ which is rather uncommon in the study of functional inequalities.

The aim of this paper is to put forward this framework  and to investigate inequalities
of the type \eqref{HSI}, \eqref{W2} and \eqref{FP} on general Riemannian manifolds.
It should be stressed that in our approach explicit Hessian estimates of the heat semigroup take over the
role of bounds on $\GG_3$.
Our results on Riemannian manifolds include the above inequalities as special cases.

We start with some basic notations.
Let $M$ be a complete connected Riemannian manifold equipped with a probability measure
$$\mu(\vd x)=\e^{-V(x)}{\vol}(\vd x)$$ for some $V\in C^2(M)$,
where $\vol(\vd x)$ denotes the Riemannian volume measure.  As well
known, the diffusion semigroup $P_t=\e^{\frac12 tL}$ generated by
$L:=\DD+\nn V$ is symmetric on $L^2(\mu)$. We denote by $\Ric_V:=\Ric+\Hess_V$ the
Bakry-\'Emery curvature tensor.

Let
$H(\nu\,|\,\mu)$, $I(\nu\,|\,\mu)$, $\W_2(\nu,\mu)$ and $S(\nu\,|\,\mu)$
for $\nu\in \scr P(M)$ be defined as in \eqref{D1}, \eqref{D2},
\eqref{D3} and \eqref{D4} respectively, with $(M,\mu)$ replacing
$(\R^n,\gg)$, the Riemannian distance $\rr(x,y)$ replacing $|x-y|$,
and $\S_\nu$ being the class of measurable 2-tensors $\tau$ which are
locally integrable with respect to $\nu$ such that
$$\int_M \<\nn V, \nn f\>\,\d\nu = \int_M \<\tau, \Hess_f\>_\HS \,\d\nu,\quad f\in C_0^\infty(M).$$
Assume $\S_{\nu}$ is non-empty, that is a Stein kernel for $\nu$  exists.
In the Euclidean case $M=\R^n$, this is ensured by the existence of a spectral gap (see \cite{CFP19}).  Existence of a Stein kernel on a general Riemannian manifold is currently work under development and will be published elsewhere.

Our results on Riemannian manifolds are presented in the Sections 3, 4
and 5. The estimates take the most concise form in case when the function $V$ satisfies $\Hess_V=K$ for some constant $K>0$.
In this case, for instance, we obtain inequalities of the same form as
in the Euclidean case:
\beg{align*}& H(\nu\,|\,\mu)\le \ff 1 2 S^2(\nu\,|\,\mu) \log\bigg(1+ \ff{I(\nu\,|\,\mu)}{KS^2(\nu\,|\,\mu)}\bigg),\\
& \W_2(\nu\,|\,\mu)\leq \ff{S(\nu\,|\,\mu)}{K^{1/2}}
\arccos\l(\exp\l(-\frac{H(\nu\,|\,\mu)}{S^2(\nu\,|\,\mu)}\r)\r),\ \
\nu\in \scr P(M),\end{align*} and there exists a constant $C>0$ such
that
 $$\l(\int |f|^p\, \vd \nu\r)^{1/p}\leq C\l(S_p(\nu\,|\,\gg)+\sqrt{p}\Big(\int |\tau|_{\rm op}^{p/2}\vd \nu\Big)^{1/p}\r),\quad \nu(f)=0,\ |\nn f|\le 1,\ \tau\in \S_\nu.$$

 The remainder of this paper is organized as follows. In Section 2 we
 study Hessian estimates for $P_t$ following the lines of
 \cite{Wang19}. Such estimates which are interesting in themselves,
 serve as crucial tools for extending \eqref{HSI}, \eqref{W2} and
 \eqref{FP} to the general geometric setting in Sections 3, 4 and 5
 respectively.  We work out some examples in Section 3.1.

\section{Hessian estimate of $P_t$}\label{section-hessian}
Let $(M,g)$ be a $n$-dimensional complete Riemannian manifold.  We
write $\langle u,v \rangle=g(u,v)$ and
$|u|=\sqrt{\langle u,u \rangle}$ for $u,v\in T_xM$ and $x\in M$. Let
$R$, $\Ric$ be the Riemann curvature tensor and Ricci curvature tensor
respectively. Recall that
$R\in \Gamma(T^*M\otimes T^*M\otimes T^*M\otimes TM)$ where
$$R(X,Y,Z)\equiv R(X,Y)Z=\nabla_X\nabla_YZ-\nabla_Y\nabla_XZ -\nabla_{[X,Y]}Z,\quad
X,Y,Z \in \Gamma(TM),$$ and $\Ric\in \Gamma(T^*M\otimes T^*M)$ given
as $\Ric(Y,Z)=\tr \big(X\mapsto R(X,Y)Z\big)$.

\begin{enumerate}[1.]
\item For $f,h\in C^2(M)$ and $x\in M$, we consider the
  Hilbert-Schmidt inner product of the Hessian tensors $\Hess_f$ and
  $\Hess_h$, i.e.
  \begin{align*}
    \langle \Hess_f, \Hess_h\rangle _\HS=\sum_{i,j=1}^n \Hess_f(X_i,X_j)\Hess_h(X_i,X_j),
  \end{align*}
  where $(X_i)_{1\leq i\leq n}$ denotes an orthonormal base of
  $T_xM$. Then the Hilbert-Schmidt norm of $\Hess_f$ is given by
  \begin{align*}
    |\Hess_f|_\HS(x)=\sqrt{\langle \Hess_f,\Hess_f \rangle_\HS}.
  \end{align*}

\item For a symmetric 2-tensor $T$ and a constant $K$, we write
  $T\geq K$ if
  \begin{align*}
    T(w,w)\geq K|w|^2,\quad w\in T_xM,\ x\in M,
  \end{align*}
  and $T\leq K$ if
  \begin{align*}
    T(w,w)\leq K|w|^2,\quad w\in T_xM,\ x\in M.
  \end{align*}

\item Given a symmetric 2-tensor $T$, we let
  $T^{\sharp}\colon TM\rightarrow TM$ be defined by
  \begin{align*}
    \langle T^\sharp(v),w \rangle=T(v,w),\quad v,w\in T_xM,\, x\in M.
  \end{align*}
  Then $T^\sharp$ is a symmetric endomorphism, i.e.,
  $\langle T^\sharp(w),v \rangle=\langle T^\sharp(v), w \rangle$ for
  $v,w \in T_xM$, $x\in M$. Let
  \begin{align*}
    |T|(x)=\sup\left\{|T^\sharp(w)|\colon w\in T_xM,\ |w|\leq 1\right\},\quad x\in M.
  \end{align*}
  Then, in particular, $|\Hess_f|(x)$ gives the operator norm of the
  Hessian of a function $f$ at $x$.

\item Furthermore, denoting by Bil$(TM)$ the vector bundle of bilinear
  forms on $TM$, we consider
  $\tilde{R}\in\Gamma(T^*M\otimes T^*M\otimes \text{Bil}(TM))$ given
  by
$$\tilde{R}(v_1,v_2)=\langle R(\boldsymbol\cdot,v_1)v_2,\boldsymbol\cdot\rangle,\quad v_1,v_2\in T_xM,$$
and let 
\begin{align*}
  &|\tilde{R}|(x)=\left||\tilde{R}(\boldsymbol\cdot,\boldsymbol\cdot)|_\HS\right |_{\HS}(x)\ \mbox{for}\ x\in M \ \ \mbox{and}\ \ \|\tilde{R}\|_{\infty}=\sup_{x\in M}|\tilde{R}|(x).
\end{align*}
Note that in explicit terms
\begin{align*}
  \|\tilde{R}\|_{\infty}=\sup_{x\in M}\left(\sum_{k,\ell}\sum_{i,j}\langle R(e_i,v_k)v_\ell,e_j\rangle^2
  \right)^{1/2}
\end{align*}
where $(v_k)_{1\leq k\leq n}$ and $(e_i)_{1\leq i\leq n}$ denote
orthonormal bases for $T_xM$.

\item For a general symmetric 2-tensor $T$, we adopt the notation
  \begin{align*}
    (RT)(v_1, v_2):=\tr\,\langle R(\boldsymbol\cdot, v_1)v_2, T^\sharp(\boldsymbol\cdot) \rangle =\sum_{i=1}^n\langle R(e_i, v_1)v_2,T^{\sharp}(e_i) \rangle,
  \end{align*}
  where $v_1, v_2\in T_xM,\ x\in M$ and $(e_i)_{1\leq i\leq n}$ is an
  orthonormal base of $T_xM$. Let
  \begin{align*}
    &|R|(x)=\sup \Big\{|(RT)(v_1,v_2)|\colon |v_1|\leq 1,\ |v_2|\leq 1,\ |T|\leq 1 \Big\} \quad \mbox{and}\quad  \|R\|_{\infty}=\sup_{x\in M}|R|(x).
  \end{align*}
It is easy to see that $|\tilde{R}|(x)\leq n |R|(x)$. In particular, if $\|R\|_{\infty}<\infty$ then $\|\tilde{R}\|_{\infty}<\infty$ as well.

\item In addition, let
$$\vd^*R=-\tr\,\nabla{\boldsymbol{.}}\,R,$$ i.e.,
$$(\vd^*R)(v_1,v_2)=-\tr\,\nabla{\boldsymbol{.}}\,R(\boldsymbol\cdot,v_1)v_2,\quad v_1,v_2\in T_xM.$$
Note that
\begin{align*}
  \langle (\vd^*R)(v_1,v_2),v_3\rangle=\langle (\nabla_{v_3}\Ric^{\sharp})(v_1),v_2 \rangle-\langle(\nabla_{v_2}\Ric^{\sharp})(v_3),v_1\rangle,\quad v_1,v_2,v_3\in T_xM.
\end{align*}
\item Finally, for $v,w \in T_xM$, let
  \begin{align*}
    R(\nabla V)(v,w):=R(\nabla V,v)w.
  \end{align*}
\end{enumerate}

In this section, we develop explicit Hessian estimates for the
semigroups which are derived from the second order derivative formula of the
semigroup obtained by first identifying appropriate local martingales.
Actually, the martingale approach to derivative formulas was first
developed by Elworthy and Li \cite{EL94}, after which an approach
based on local martingales has been worked out by Thalmaier
\cite{Thalmaier97} and Driver and Thalmaier \cite{DTh2001}. Although
various formulas for the Hessian appear in the literature, for
example \cite{APT, EL94, Li, Wang19, Th19, Th2020}, Hessian estimates
of the heat semigoup are not well calculated with explicit constants
depending on the curvature tensor on general Riemannian manifolds. Our
Theorems \ref{general-hessian-theorem} and
\ref{general-hessian-theorem2} fill this gap and are new in this regard.

\subsection{Hessian estimates of semigroup: type I }
Let us introduce a first type of Hessian estimate of the heat semigroup.  When $M$ is Ricci parallel and the generator of the
diffusion equals half the Laplacian~$\Delta$, such a type of formula
bounding the norm of the Hessian of $P_tf$ from
above by $P_t|\nabla f|^2$, has been already given in \cite{Wang19}.

\begin{theorem}[Hessian estimate: type I]\label{general-hessian-theorem}
  Assume that $\Ric_V\geq K$, $\|R\|_{\infty}<\infty$ and
$$\beta:=\|\nabla \Ric_V^{\sharp}+\vd^*R+R(\nabla V)\|_{\infty}<\infty.$$
Let $\alpha_1:=\|R\|_{\infty}$ and $\alpha_2:=\|\tilde{R}\|_{\infty}$. Then for $f\in C_b^2(M)$,
\begin{align*}
  &|\Hess_{P_tf}|\\
  &\leq \l(\frac{K-2\alpha_1}{\e^{(2K-2\alpha_1)t}-\e^{Kt}}\r)^{1/2}\l((P_t|\nabla f|^2)^{1/2}+ \l(\frac{\e^{Kt}-1}{K}\r)^{1/2}\frac{\beta}{K}(P_t|\nabla f|)\r).
\end{align*}
Moreover, if $\Ric_V=K$, then
\begin{align}
  &|\Hess_{P_tf}|_{\HS}\notag\\
  &\leq \l(\frac{K-2\alpha_2}{\e^{(2K-2\alpha_2)t}-\e^{Kt}}\r)^{1/2}\l((P_t|\nabla f|^2)^{1/2}+ \l(\frac{\e^{Kt}-1}{K}\r)^{1/2}\frac{n\beta}{K}(P_t|\nabla f|)\r).\label{esti-HS-ineq}
\end{align}
\end{theorem}

To prove Theorem \ref{general-hessian-theorem}, we first introduce a
probabilistic representation formula for $\Hess_{P_tf}$. For the
semigroup $P_t$ generated by $\Delta/2$, a Bismut type Hessian formula
has been established in~\cite{APT}, which was then extended to general
Schr\"{o}dinger operators on $M$ \cite{Li, Th19}.

Denote by $\Ric^{\sharp}_V=\Ric^{\sharp}+\Hess_V^{\sharp}$ the
Bakry-\'Emery tensor (written as endomorphism of $TM$).  The damped
parallel transport $Q_t\colon T_xM\rightarrow T_{X_t}M$ is defined as
the solution, along the paths of $X_t$, to the covariant ordinary
differential equation
\begin{align*}
  D Q_t=-\frac12 \Ric^{\sharp}_VQ_t\, \vd t,\quad Q_0=\id,
\end{align*}
where the covariant differential is given by $\ptr_t^{-1}\,D=\vd\,\ptr_t^{-1}$.

For $w\in T_xM$, we define an operator-valued
process $W_t(\boldsymbol\cdot, w): T_{x}M\rightarrow T_{X_t}M$ by
\begin{align*}
  W_t(\boldsymbol\cdot, w):=&Q_t\int_0^tQ_r^{-1}R(\ptr_r \, \vd B_r, Q_r(\boldsymbol\cdot))Q_r(w)\\
                 &-\frac12 Q_t\int_0^tQ_r^{-1}\big(\nabla \Ric_V^{\sharp}+\vd^*R+R(\nabla V)\big)\big(Q_r(\boldsymbol\cdot), Q_r(w)\big)\,\vd r.
\end{align*}
Note that $W_t(\boldsymbol\cdot, w)$ is the solution to the covariant It\^{o}
equation
\begin{align*}
  DW_t(\boldsymbol\cdot, w)
  &=R(\ptr_t\, \vd B_t, Q_t(\boldsymbol\cdot)) Q_t(w)-\frac12 \Ric_V^{\sharp}(W_t(\boldsymbol\cdot, w))\, \vd t\\
  &\quad-\frac12(\vd^*R+\nabla\Ric_V^{\sharp}+R(\nabla V))(Q_t(\boldsymbol\cdot), Q_t(w))\,\vd t,
\end{align*}
with initial condition $W_0(\boldsymbol\cdot, w)=0$.

\begin{lemma}\label{Hess-prop-1}
 Let $\rho$ be the Riemannian distance to a fixed point $o\in M$.
  Assume that
  \begin{align*}
    \lim_{\rho\rightarrow \infty}\frac{\log \Big(\big|\vd^*R+\nabla\Ric_V^{\sharp}+R(\nabla V)\big|+|R|\Big)}{\rho^2}=0,
  \end{align*}
  and
  \begin{align*}
    \Ric_V\geq -h(\rho)\quad \mbox{for some positive function } h\in C([0,\infty)) \mbox{ such that }\ \lim_{r\rightarrow \infty}\frac{h(r)}{r^2}=0.
  \end{align*}
  Then
  \begin{align*}
    \Hess_{P_tf}(v,w)=\E\l[\Hess_f(Q_t(v),Q_t(w))+\langle \nabla f(X_t), W_t(v,w)\rangle\r].
  \end{align*}
\end{lemma}

\begin{proof}
  For fixed $T>0$, set
  \begin{align*}
    N_t(v,w):= \Hess_{P_{T-t}f}(Q_t(v), Q_t(w))+\langle\nabla P_{T-t}f(X_t), W_t(v,w)\rangle.
  \end{align*}
  We first recall that $N_t(v,w)$ is a local martingale, which has been shown  e.g. in \cite[Lemma 11.3]{Th2020}. We include a proof here for the convenience of the reader.
  We first observe that
  \begin{align*}
    \vd(\Delta-\nabla V)f&=\left(\tr \nabla^2-\nabla _{\nabla V}\right)  \vd f-\vd f(\Ric_V^{\sharp}),\\
    \nabla \vd (\Delta f)&=\tr \nabla^2(\nabla \vd f)-(\nabla \vd f)(\Ric^{\sharp}\odot\id+\id\odot\Ric^{\sharp}-2R^{\sharp, \sharp})-\vd f(\vd ^*R+\nabla \Ric^{\sharp}),\\
    \nabla \vd (\nabla V(f))&=\nabla_{\nabla V}(\nabla \vd f)+(\nabla \vd f)(\Hess_V^{\sharp}\odot\id+\id\odot\Hess_V^{\sharp})+\vd f(\nabla \Hess_V^{\sharp} +R(\nabla V)),
  \end{align*}
  where $\odot$ denotes the symmetric tensor product.
  Thus, for the It\^o differential of $N_t(v,w)$, we obtain
  \begin{align*}
    \vd N_t(v,w)&=(\nabla_{\ptr_t\,\vd B_t}\Hess_{P_{T-t}f})(Q_t(v), Q_t(w))+\Hess_{P_{T-t}f}\l(\frac{D }{\vd  t} Q_t(v), Q_t(w)\r)\vd t\\
                &\quad+\Hess_{P_{T-t}f}\l(Q_t(v),\frac{D}{\vd t} Q_t(w)\r)\vd t+\partial_t(\Hess_{P_{T-t}f})(Q_t(v),Q_t(w))\,\vd t
    \\
                &\quad+\frac12\tr(\nabla^2-\nabla_{\nabla V})(\Hess_{P_{T-t}f})(Q_t(v),Q_t(w))\,\vd t+(\nabla_{\ptr_t\, \vd B_t}\vd P_{T-t}f)(W_t(v,w))\\
                &\quad+(\vd P_{T-t}f)(DW_t(v,w))+\langle D(\vd P_{T-t}f), DW_t(v,w)\rangle+\partial_t(\vd P_{T-t}f)(W_t(v,w))\,\vd t\\
                &\quad+\frac12\tr(\nabla^2-\nabla_{\nabla V})(\vd P_{T-t}f)(W_t(v,w))\, \vd t\\
                &\mequal-\frac12\Hess_{P_{T-t}f}\l(\Ric_V^{\sharp}(Q_t(v)), Q_t(w)\r)\vd t-\frac12\Hess_{P_{T-t}f}\l(Q_t(v),\Ric_V^{\sharp} (Q_t(w))\r)\vd t\\
                &\quad-\frac12(\nabla \vd (\Delta-\nabla V){P_{T-t}f})(Q_t(v),Q_t(w))\,\vd t
                  +\frac12\left(\tr\nabla^2-\nabla_{\nabla V}\right)(\Hess_{P_{T-t}f})(Q_t(v),Q_t(w))\,\vd t\\
                &\quad-\frac12(\vd P_{T-t}f)(\vd^*R+\nabla\Ric_V^{\sharp}+R(\nabla V))(Q_t(v), Q_t(w))\,\vd t\\
                &\quad-\frac12(\vd P_{T-t}f)(\Ric_V^{\sharp}(W_t(v, w)))\, \vd t+\tr\l\{\Hess_{P_{T-t}f}(\boldsymbol\cdot, R(\boldsymbol\cdot, Q_t(v))Q_t(w))\r\}\,\vd t\\
                &\quad-\frac12(\vd (\Delta-\nabla V)P_{T-t}f)(W_t(v,w))\,\vd t+\frac12\left(\tr\nabla^2-\nabla_{\nabla V}\right)(\vd P_{T-t}f)(W_t(v,w))\, \vd t\\
                &=0,
  \end{align*}
  where $\mequal$ denotes equality modulo differentials of local martingales,
  so that $N_t$ is a local martingale.  Assume that
  \begin{align*}
    \lim_{\rho\rightarrow \infty}\frac{\log \l(|\vd^*R+\nabla\Ric_V^{\sharp}+R(\nabla V)|+|R|\r)}{\rho^2}=0,
  \end{align*}
  and
  \begin{align*}
    \Ric_V\geq -h(\rho)\quad \mbox{for some positive } \ h\in C([0,\infty))\ \mbox{with}\ \ \lim_{r\rightarrow \infty}\frac{h(r)}{r^2}=0.
  \end{align*}
  Then by \cite[Proposition 3.1]{Wang19}, for $t>0$ we have
  \begin{align*}
    &\E\left[\sup_{s\in [0,t]}|Q_s|^2\right]<\infty \quad \mbox{and}\quad \E\left[\sup_{s\in [0,t]}|W_s|^2\right]<\infty.
  \end{align*}
  In addition, $|\nabla P_{T-t}|(x)$ and $|\Hess _{ P_{T-t}}|(x)$ are easy to bound
  by local Bismut type formulae \cite{APT, Th19}.
  Under our curvature assumptions these local bounds then provide global bounds
  uniformly in $(t,x)\in {[0,T-\varepsilon]}\times M$ for every small $\varepsilon>0$. 
  Thus the local martingale $N_t$ is a true martingale on the time interval
  $[0,T-\varepsilon]$. By taking expectations, we first obtain $\E[N_0]=\E[N_{T-\varepsilon}]$
  and then
  \begin{equation*}
    \Hess_{P_Tf}(v,w)=\E\l[\Hess_f(Q_T(v),Q_T(w))+\langle \nabla f(X_T), W_T(v,w)\rangle\r]
  \end{equation*}
  by passing to the limit as $\varepsilon\downarrow0$.
  Note that since the manifold is complete,
  we have by the spectral theorem $\vd P_tf=P_t\vd f$ where $P_tdf(v)=\E[(\vd f)(X_t)Q_tv]$ is the
  canonical heat semigroup on $1$-forms (see \cite{DTh2001}).
\end{proof}

According to the definition of $W_t$, we have
\begin{align*}
\E\langle \nabla f(X_t), W_t(v,w) \rangle= &\E\Big\langle \nabla f(X_t), Q_t\int_0^tQ_r^{-1}R(\ptr_r\,\vd B_r, Q_r(v))Q_r(w) \Big \rangle \\
&-\frac12 \E  \Big\langle \nabla f(X_t), Q_t\int_0^tQ_r^{-1}\big(\nabla \Ric_V^{\sharp}+\vd^*R+R(\nabla V)\big)\big(Q_r(v), Q_r(w)\big)\,\vd r \Big \rangle.
\end{align*}
To deal with the first term on the right hand side, we observe that

\begin{lemma}\label{lem-1}
  Keeping the assumptions of Lemma \ref{Hess-prop-1}, we have
\begin{align*}
    \E\left[\Big\langle\nabla f(X_t), Q_t\int_0^tQ_r^{-1}R(\ptr_r\,\vd B_r, Q_r(v))Q_r(w)\Big\rangle\right]=\E\left[\int_0^t (R\Hess_{P_{t-s}f}) (Q_s(v),Q_s(w)\,\vd s\right].
\end{align*}
\end{lemma}

\begin{proof}

  Let
  $$H_s(v,w)=\Big\langle \nabla P_{t-s}f(X_s), Q_s\int_0^s
  Q_r^{-1}R(\ptr_r\,\vd B_r, Q_r(v))Q_r(w)\Big\rangle.$$ It is easy to
  check that
  \begin{align*}
    \vd (H_s(v,w))&=\Big\langle \nabla_{\ptr_s\vd B_s}(\nabla P_{t-s}f)(X_s), Q_s\int_0^s Q_r^{-1}R(\ptr_r\,\vd B_r, Q_r(v))Q_r(w) \Big\rangle \\
                  &\quad+ \Big\langle \Ric_V^{\sharp}(\nabla P_{t-s}f)(X_s), Q_s\int_0^s Q_r^{-1}R(\ptr_r\,\vd B_r, Q_r(v))Q_r(w)\Big\rangle\,\vd s\\
                  &\quad-\Big\langle (\nabla P_{t-s}f)(X_s), \Ric_V^{\sharp}\Big(Q_s\int_0^s Q_r^{-1}R(\ptr_r\,\vd B_r, Q_r(v))Q_r(w)\Big)\Big\rangle\,\vd s\\
                  &\quad+\Big\langle (\nabla P_{t-s}f)(X_s), R(\ptr_s\,\vd B_s, Q_s(v))Q_s(w) \Big\rangle\\
                  &\quad+\tr\,\langle \nabla\boldsymbol{.}(\nabla P_{t-s}f),  R(\boldsymbol\cdot, Q_s(v))Q_s(w)\rangle \,\vd s\\
                  &\mequal\tr\l(\Hess_{P_{t-s}f}(\boldsymbol\cdot, R(\boldsymbol\cdot, Q_s(v))Q_s(w))\r)\vd s
  \end{align*}
  which implies 
  \begin{equation*}
    \E\left[\Big\langle\nabla f(X_t), Q_t\int_0^tQ_s^{-1}R(\ptr_s\,\vd B_s, Q_s(v))Q_s(w)\Big\rangle\right]\
    =\E\left[\int_0^t \tr\l(\Hess_{P_{t-s}f}(\boldsymbol\cdot, R(\boldsymbol\cdot, Q_s(v))Q_s(w))\r)\vd s\right].
  \end{equation*}
\end{proof}

With these two lemmas we are now in position to prove Theorem
\ref{general-hessian-theorem}.

\begin{proof}[Proof of Theorem \ref{general-hessian-theorem}]
  We begin with the following observation obtained by combining the
  formulas in Lemmas \ref{Hess-prop-1} and \ref{lem-1}:
  \begin{align*}
    \Hess_{P_tf}(v,w)&=\E\left[\Hess_{f}(Q_t(v), Q_t(w))\right]+\E\left[\langle\nabla f(X_t), W_t(v,w) \rangle\right]\\
                     &=\E\left[\Hess_f(Q_t(v), Q_t(w))\right]+\E\left[\Big\langle\nabla f(X_t),Q_t\int_0^tQ_r^{-1}R(\ptr_r\vd B_r, Q_r(v))Q_r(w)\Big \rangle\right]\\
                     &\quad -\frac12\E\left[\Big\langle\nabla f(X_t),Q_t\int_0^tQ_r^{-1}(\nabla \Ric_V^{\sharp}+\vd^*R+R(\nabla V))(Q_r(v), Q_r(w))\,\vd r \Big\rangle\right]\\
                     &=\E\left[\Hess_f(Q_t(v), Q_t(w))\right]+\E\left[\int_0^t\tr\left(\Hess_{P_{t-s}f}(\boldsymbol\cdot, R(\boldsymbol\cdot, Q_s(v))Q_s(w))\right)\vd s\right]\\
                     &\quad -\frac12\E\left[\Big\langle\nabla f(X_t),Q_t\int_0^tQ_r^{-1}(\nabla \Ric_V^{\sharp}+\vd^*R+R(\nabla V))(Q_r(v), Q_r(w))\,\vd r \Big\rangle\right].
  \end{align*}
  Noting that $|Q_tQ_{r}^{-1}|\le \e^{-K(t-r)/2}$,
  $|Q_r|\le \e^{-Kr/2}$, and
  \begin{align*}
    \tr\left(\Hess_{P_{t-s}f}(\boldsymbol\cdot, R(\boldsymbol\cdot, Q_s(v))Q_s(w))\right)
    &\leq \e^{-Ks}\,|\Hess_{P_{t-s}f}|(X_{s})\, \|R\|_{\infty},
  \end{align*}
  where $(e_i)_{1\leq i\leq n}$ is an orthonormal base of $T_xM$, we
  derive
  \begin{align*}
    |\Hess_{P_tf}|
    &\leq \e^{-Kt}\,P_t|\Hess_f|+\|R\|_{\infty}\int_0^t\e^{-Ks}P_s|\Hess_{P_{t-s}f}|\,\vd s 
      +\frac \bb 2 \l( \int_0^t\e^{-K(t+r)/2} \vd r\r)  P_t|\nn f|\\
    &= \e^{-Kt}\,P_t|\Hess_f|+\ff{\bb  (\e^{-Kt/2}-\e^{-Kt}) }K P_t|\nn f|  +\|R\|_\infty\int_0^t\e^{-Ks}P_s|\Hess_{P_{t-s}f}|\,\vd s,\ \ t\ge 0. 
  \end{align*}
  Now let
$$ \phi(r):=\e^{-K(t-r)}P_{t-r}|\Hess_{P_rf}|,\ \   r\in [0,t].$$
Applying the above estimate for $P_rf$ instead of $P_tf$, and noting
that $\ff {\e^{Kr/2}-1}K$ is increasing in $r$, we obtain
\begin{align*}
  \phi(r)&\leq \phi(0)+  \bb \e^{-Kt} \ff {\e^{Kr/2}-1}K P_{t-r} (P_r|\nn f| ) +\|R\|_\infty\int_0^r\phi(r-s)\,\vd s\\
         &\le \phi(0)+ \ff{\bb(\e^{-Kt/2}-\e^{-Kt})}K P_t|\nn f|   + \|R\|_\infty\int_0^r\phi(s)\,\vd s,\ \ r\in [0,t].  \end{align*}
       By Gronwall's lemma, this implies 
       \begin{equation}\label{W1} \begin{split}
           |\Hess_{P_tf}|&= \phi(t)\leq \bigg\{\phi(0)+ \ff{\bb(\e^{-Kt/2}-\e^{-Kt})}K P_t|\nn f|\bigg\} \,\e^{\|R\|_\infty t}\\
           &=\e^{(\|R\|_\infty-K)t}P_t|\Hess_f| +
           \ff{\bb\,\e^{\|R\|_\infty
               t}\big(\e^{-Kt/2}-\e^{-Kt}\big)}K P_t|\nn f|.
         \end{split}\end{equation}
       On the other hand, by It\^{o}'s formula we have
       \begin{align*}
         \vd |\nabla P_{t-s}f|^2(X_s)&=\frac12\l(L|\nabla P_{t-s}f|^2(X_s)-\langle\nabla P_{t-s}f,\nabla LP_{t-s}f \rangle(X_s)\r)\vd s\\
                                     &\quad+\langle \nabla |\nabla P_{t-s}f|^2(X_s),\ptr_s \vd B_s \rangle,\quad s\in [0,t].
       \end{align*}
       Using the Bochner-Weitzenb\"{o}ck formula and the assumption
       $\Ric_V\geq K$, we obtain
       \begin{align*}
         &  \vd |\nabla P_{t-s}f|^2(X_s)\\
         &\geq \l(\Ric_V(\nabla P_{t-s}f,\nabla P_{t-s}f)+|\Hess_{P_{t-s}f}|^2_\HS\r)(X_s)\,\vd s
           +\langle \nabla |\nabla P_{t-s}f|^2(X_s),\ptr_s \vd B_s \rangle\\
         &\geq  K|\nabla P_{t-s}f|^2(X_s)\,\vd s
           +|\Hess_{P_{t-s}f}|^2_\HS(X_s)\,\vd s
           +\langle \nabla |\nabla P_{t-s}f|^2(X_s),\ptr_s \vd B_s \rangle.
       \end{align*}
       From this, we conclude that
       \begin{align*}
         P_t|\nabla f|^2-\e^{Kt}|\nabla P_tf|^2\geq \int_0^t\e^{K(t-s)}P_s|\Hess_{P_{t-s}f}|_\HS^2\, \vd s.
       \end{align*}
       By the inequalities of Jensen and Schwartz, this yields
       \beg{align*}
        \e^{-Kt/2}  (P_t|\nabla f|^2)^{1/2} &\ge \bigg(\int_0^t \e^{-2\|R\|_{\infty}s} \e^{(2\|R\|_{\infty}-K)s} (P_s|\Hess_{P_{t-s}f}|_\HS)^2\,\d s\bigg)^{1/2}\\
       & \ge \left(\ff
       {K-2\|R\|_{\infty}}{\e^{(K-2\|R\|_{\infty})t}-1}
       \right)^{1/2} \int_0^t \e^{-\|R\|_{\infty}s}
       P_s|\Hess_{P_{t-s}f}|_{\HS} \,\d s.\end{align*}
     Combining this
     with \eqref{W1} for $(P_s, P_{t-s}f)$ instead of  $(P_t, f)$, and
     noting that $|\nn P_{t-s}f|\le \e^{-K(t-s)/2}P_{t-s}|\nn f|$, we
     arrive at
  \begin{align*}
& \e^{-Kt/2}\left(\ff {\e^{(K-2\|R\|_{\infty})t}-1} {K-2\|R\|_{\infty}} \right)^{1/2}(P_t|\nn f|^2)^{1/2}\\
  &\ge \int_0^t \e^{-\|R\|_{\infty}s} P_s|\Hess_{P_{t-s}f}|_{\HS} \d s \\
  &\ge \int_0^t \e^{-\|R\|_{\infty}s} \left(\e^{(K-\|R\|_\infty)s}|\Hess_{P_tf}| -\ff{\bb\,\e^{Ks}(\e^{-Ks/2}-\e^{-Ks})}K P_s|\nn P_{t-s}f|\right)\d s\\
  &\ge \ff{\e^{(K-2\|R\|_\infty) t}-1}{K-2\|R\|_\infty} |\Hess_{P_tf}| - \bb (P_t|\nn f|)\,\e^{-Kt/2}\int_0^t \e^{(K-\|R\|_{\infty})s} \ff{1-\e^{-Ks/2}}K\,\d s\\
 &\ge \ff{\e^{(K-2\|R\|_\infty) t}-1}{K-2\|R\|_\infty} |\Hess_{P_tf}| - \ff{\beta}K\,\e^{-Kt/2}\l( \ff{\e^{(K-2\|R\|_\infty) t}-1}{K-2\|R\|_\infty}\r)^{1/2} \l(\frac{\e^{Kt}-1}{K}\r)^{1/2} (P_t|\nn f|).
\end{align*}
This completes the proof of the first inequality.

For the second case, when
  $\Ric_V= K$ we realize that 
  $Q_t(v)=\e^{-Kt/2}\ptr_tv$ for $v\in T_xM$, and that for all $f\in C_b^2(M)$ and
  $v,w\in T_xM$ such that $|v|=|w|=1$,
  \begin{align}
    &\Hess_{P_{t}f}(v,w)\\
    &=\E\left[\Hess_f(Q_t(v), Q_t(w))\right]+\E\int_0^t(R\Hess_{P_{s}f})(Q_{t-s}v, Q_{t-s}w)\, \vd s\notag\\
                     &\quad -\frac12\E\left[\Big\langle\nabla f(X_t),Q_t\int_0^tQ_r^{-1}(\nabla \Ric_V^{\sharp}+\vd^*R+R(\nabla V))(Q_r(v), Q_r(w))\,\vd r \Big\rangle\right]\notag\\
    &=\e^{-tK}\E\l[\Hess_f(\ptr_tv, \ptr_tw)(X_t)\r] +\ff{\bb  (\e^{-Kt/2}-\e^{-Kt}) }K P_t|\nn f| \notag\\
&\quad+ \int_{0}^{t}\E \l[\e^{-K(t-s)}(R\Hess_{P_sf})(\ptr_{t-s}v, \ptr_{t-s}w) \r]\, \vd s \notag \\
    &\leq \e^{-tK}\E\l[\Hess_f(\ptr_tv, \ptr_tw)(X_t)\r]+\ff{\bb  (\e^{-Kt/2}-\e^{-Kt}) }K P_t|\nn f|\notag\\
 &\quad +\int_{0}^{t}\E\l[\e^{-K(t-s)}\,\tr\langle \Hess_{P_sf}(\boldsymbol\cdot), R(\boldsymbol\cdot, \ptr_{t-s}v)\ptr_{t-s}w \rangle \r]\, \vd s \notag\\
    &\leq \e^{-tK}\E\l[\Hess_f(\ptr_tv, \ptr_tw)(X_t)\r]+\ff{\bb  (\e^{-Kt/2}-\e^{-Kt}) }K P_t|\nn f| \notag\\
&\quad +\int_{0}^{t}\e^{-K(t-s)}\E\l[\big|\Hess_{P_s f}\big|_\HS(X_{t-s})\, |\tilde{R}(\ptr_{t-s}v,\ptr_{t-s}w)|_{\HS}(X_{t-s})\r]\, \vd s.\label{est:Hess-1} 
  \end{align}
  This gives us
  \begin{align*}
   & |\Hess_{P_{t}f}|_\HS\\
&\leq 
      \E\l[\e^{-K t}|\Hess_{f}|_\HS(X_t)\r]+ \ff{n\,\bb  (\e^{-Kt/2}-\e^{-Kt}) }K P_t|\nn f|+{}\\
& +\sqrt{ \sum_{i,j}\l(\E\left[\int_{0}^{t}\e^{-K(t-s)}\l(\big|\Hess_{P_s f}\big|_\HS(X_{t-s})\, |\tilde{R}(\ptr_{t-s}e_i,\ptr_{t-s}e_j)|_{\HS}(X_{t-s})\r)\vd s\r]\r)^2 }\\
&\leq 
      \E\l[\e^{-K t}|\Hess_{f}|_\HS(X_t)\r]+ \ff{n\,\bb  (\e^{-Kt/2}-\e^{-Kt}) }K P_t|\nn f|+{}\\
& +\sqrt{ \sum_{i,j}\E\l[\int_{0}^{t}\e^{-K(t-s)}\big|\Hess_{P_s f}\big|_\HS(X_{t-s}) \,\vd s\r]\E \l[\int_{0}^{t}\e^{-K(t-s)}\l(\big|\Hess_{P_s f}\big|_\HS\, |\tilde{R}(\ptr_{t-s}e_i,\ptr_{t-s}e_j)|^2_{\HS}\r)(X_{t-s})\, \vd s\r] }\\
&\leq 
      \E\l[\e^{-K t}|\Hess_{f}|_\HS(X_t)\r]+ \ff{n\bb  (\e^{-Kt/2}-\e^{-Kt}) }K P_t|\nn f| +\|\tilde{R}\|_{\infty}\,{\E\l[\int_{0}^{t}\e^{-K(t-s)}\big|\Hess_{P_s f}\big|_\HS(X_{t-s}) \,\vd s\r]}.
  \end{align*}
  The remaining steps are similar to the first part of the proof; we skip the details.
\end{proof}

Important examples in the sequel will be Ricci parallel manifolds
which is the class of Riemannian manifolds where Ricci curvature is
constant under parallel transport, that is $\nabla \Ric=0$ for the
Levi-Civita connection $\nabla$.  Recall that an Einstein manifold is
Ricci parallel but in general the inverse is not true.

Recently F.-Y. Wang \cite{Wang19} used functional inequalities for the
semigroup to characterize constant curvature manifolds, Einstein
manifolds, and Ricci parallel manifolds.  Here we list the results
for the Hessian estimate of $P_t$ generated by the operator
$\frac12 L$ when  $M$ is a Ricci parallel manifold and $\nabla V$ is
a Killing field on $(M,g)$.  Here, a vector field $X$ on a Riemannian
manifold $(M,g)$ is called a {\it Killing field} if the local flows
generated by $X$ act by isometries i.e., for $Y,Z\in TM$,
\begin{align*}
  \nabla ^2_{Y,Z}(X)=-R(X, Y)Z.
\end{align*}
We conclude that
$\|\vd^*R+\nabla\Ric_V^{\sharp}+R(\nabla V)\|_{\infty}=0$ if
$\nabla V$ is a Killing field on a Ricci parallel manifold $(M,g)$.

\begin{corollary}\label{esti-Hess}
  Assume that $M$ is a Ricci parallel manifold, $\nabla V$ is a
  Killing field and $\|R\|_{\infty}<\infty$.  Then for any constant
  $K\in \R$,
  \begin{enumerate}[\rm(i)]
  \item if\/ $\Ric_V\geq K>0$, then for any $f\in C_b^2(M)$ and
    $t\geq 0$,
    \begin{align*}
      |\Hess_{P_tf}|_\HS^2\leq
      \frac{n(2\|R\|_{\infty}-K)}{\e^{Kt}-\e^{2(K-\|R\|_{\infty})t}} P_t|\nabla f|^2;
    \end{align*}
  \item if\/ $\Ric_V=K>0$, then for any $f\in C_b^2(M)$ and $t\geq 0$,
    \begin{align*}
      |\Hess_{P_tf}|_\HS^2\leq
      \frac{2\|\tilde{R}\|_{\infty}-K}{\e^{Kt}-\e^{2(K-\|\tilde{R}\|_{\infty})t}} P_t|\nabla f|^2.
    \end{align*}
  \end{enumerate}
\end{corollary}

\begin{proof}
  These items are direct consequences of Theorem
  \ref{general-hessian-theorem}. The second assertion can also be
  proved by an argument as in \cite[Theorem 4.1]{Wang19} with some
  straightforward modifications.
\end{proof}

\subsection{Hessian estimate of semigroup: type II}
We now introduce a slightly different type of Hessian estimate for the
semigroup.
\begin{theorem}[Hessian estimate: type
  II]\label{general-hessian-theorem2}
  Assume that $\Ric_V\geq K>0$, $\alpha_1:=\|R\|_{\infty}<\infty$ (or
  $\alpha_2:=\|\tilde{R}\|_{\infty}<\infty$ ) and
$$\beta:=\|\nabla \Ric_V^{\sharp}+\vd^*R+R(\nabla V)\|_{\infty}<\infty.$$
Then for $f\in C_b^2(M)$,
\begin{align*}
  |\Hess _{P_tf}|\leq & \l(\frac{\e^{-Kt/2}}{\sqrt{\int_0^t\e^{Kr}\,\vd r}}+\frac{\alpha_1\e^{-Kt/2}}{\sqrt{K}}\r)(P_t|\nabla f|^2)^{1/2}+ \frac{\beta \e^{-Kt/2}}{K}(P_t|\nabla f|).
\end{align*}
Moreover, if $\Ric_V=K$, then for $f\in C_b^2(M)$,
\begin{align*}
  |\Hess_{P_tf}|_{\rm HS}\leq & \l(\frac{\e^{-Kt/2}}{\sqrt{\int_0^t\e^{Kr}\,\vd r}}
                                +\frac{\alpha_2\e^{-Kt/2}}{\sqrt{K}}\r)(P_t|\nabla f|^2)^{1/2}+ \frac{ n \beta \e^{-Kt/2}}{K}(P_t|\nabla f|).
\end{align*}
\end{theorem}

To prove this theorem, we need the following Hessian and gradient
formula for the semigroup which is similar to \cite[Theorem
11.6]{Th2020} with the difference in the use of
$W^k(\boldsymbol\cdot, \boldsymbol\cdot)\colon TM\times TM\rightarrow M$.

\begin{lemma}\label{Hess-prop-1}
 Let $\rho$ be the Riemannian distance to a fixed point $o\in M$.
  Assume that
  \begin{align*}
    \lim_{\rho\rightarrow \infty}\frac{\log \Big(\big|\vd^*R+\nabla\Ric_V^{\sharp}+R(\nabla V)\big|+|R|\Big)}{\rho^2}=0,
  \end{align*}
  and
  \begin{align*}
    \Ric_V\geq -h(\rho)\quad \mbox{for some positive function } h\in C([0,\infty)) \mbox{ such that }\ \lim_{r\rightarrow \infty}\frac{h(r)}{r^2}=0.
  \end{align*}
  Then for $k\in C^1([0,t])$ with $k(0)=1$ and $k(t)=0$,
  \begin{align*}
    \Hess_{P_tf}(v,w)=\E^x\l[-\nabla f(Q_t(v))\int_0^t\langle Q_s(\dot{k}(s)w),\, \ptr_s \vd B_s\rangle+\langle \nabla f(X_t), W^k_t(v,w)\rangle\r],
  \end{align*}
  for $v,w\in T_xM$,  where
  \begin{align*}
  W_t^k(\boldsymbol\cdot, w):=&Q_t\int_0^tQ_r^{-1}R(\ptr_r \, \vd B_r, Q_r(\boldsymbol\cdot))Q_r(k(r) w)\\
                 &-\frac12 Q_t\int_0^tQ_r^{-1}\big(\nabla \Ric_V^{\sharp}+\vd^*R+R(\nabla V)\big)\big(Q_r(\boldsymbol\cdot), Q_r(k(r) w)\big)\,\vd r.
\end{align*}
\end{lemma}

\begin{proof}
  Fixed $T>0$, set
  \begin{align*}
    N_t(v,w):= \Hess_{P_{T-t}f}(Q_t(v), Q_t(w))+\langle\nabla P_{T-t}f(X_t), W_t(v,w)\rangle.
  \end{align*}
 Furthermore, define
$$N^k_t(v,w)=\Hess_{P_{T-t}f}(Q_t(v), Q_t(k(t)w))+(\vd P_{T-t}f)(W^k_t(v,w)).$$ 
According to the definition of $W^k_t(v,w)$, resp.~$W_t(v,w)$, and in
view of the fact that $N_t(v,w)$ is a local martingale, it is easy to
see that
\begin{align}\label{local-M1}
  N_t^k(v,w)&-\int_0^t(\Hess _{ P_{T-s}f})(Q_s(
              v), Q_s(\dot{k}(s) w))\,\vd s 
\end{align}
is a local martingale. From the formula
\begin{align*}
  \vd P_{T-t}f(Q_t(v))=\vd P_{T}f(v)+\int_0^t(\Hess _{ P_{T-s}f})
  (\ptr_s \vd  B_s, Q_s(v)),
\end{align*}
it follows that
\begin{align}\label{local-M3}
 &\int_0^t(\Hess _{ P_{T-s}f})(Q_s(v), Q_s(\dot{k}(s) w))
    \,\vd s-\vd P_{T-t}f(Q_t(v))\int_0^t\langle Q_s(\dot{k}(s)w),
    \ptr_s\vd  B_s\rangle 
\end{align}
is also a local martingale. Concerning the last term in
\eqref{local-M3}, we note that
\begin{align*}
  M_t:=\Hess _{ P_{T-t}f}(Q_t(v), Q_t(k(t)w))+(\vd P_{T-t}f)(W^k_t(v,w))-\vd P_{T-t}f(Q_t(v))\int_0^t\langle Q_s(\dot{k}(s)w),
    \ptr_s\vd  B_s\rangle 
\end{align*}
is a local martingale as well. 
  As explained in the proof of Theorem \ref{general-hessian-theorem},  the local martingale $M_t$ is a true martingale on the time interval
  $[0,T-\varepsilon]$. By taking expectations, we first obtain $\E[M_0]=\E[M_{T-\varepsilon}]$
  and then
  \begin{equation*}
    \Hess_{P_Tf}(v,w)=\E\l[-\vd f(Q_T(v))\int_0^T\langle Q_s(\dot{k}(s)w),
    \ptr_s\vd  B_s\rangle+\vd f( W_T^k(v,w))\r]
  \end{equation*}
 by passing to the limit as $\varepsilon\downarrow0$.  
\end{proof}

\begin{proof}[Proof of Theorem \ref{general-hessian-theorem2}]

As
 $\alpha_1:=\|R\|_{\infty}<\infty$, $\Ric_Z\geq K$
  for some constants $K$ and
\begin{align*}
\beta:=\|{\bf d} ^*R+\nabla \Ric_Z^{\sharp}-R(Z)\|_{\infty}<\infty,
\end{align*}
then for all $t>0$, 

\begin{align*}
&\E\l[\vd f(Q_t(v))\int_0^t\langle Q_s(\dot{k}(s)w),
    \ptr_s\vd  B_s\rangle \r] \\
& \leq \e^{-\frac{K}{2}t} ( P_t|\nabla f|^2)^{1/2}  \l(\int_0^t\e^{-Ks}\dot{k}(s)^2\,\vd s\r)^{1/2}, \\
&\E\l[\vd f \l(Q_t\int_0^tQ_r^{-1}R(\ptr_r\,\vd B_r, Q_r(k(r)w))Q_r(v)\r)\r] \\
&\leq \alpha_1  \e^{-\frac{K t}{2}} (P_t|\nabla f|^2)^{1/2}\l(\int_0^t\e^{-Ks} k(s)^2\, \vd s\r)^{1/2},
\end{align*}
and 
\begin{align*}
&\frac{1}{2}\E\l[\vd f \l( Q_t\int_0^tQ_r^{-1}\big(\nabla \Ric_V^{\sharp}+\vd^*R+R(\nabla V)\big)\big(Q_r(k(r)w), Q_r(v)\big)\,\vd r \r)\r]  \\
&\leq \frac{\beta }{2}\e^{-{K t}/{2}} \l(\int_0^t\e^{-{Ks}/{2}}k(s)\, \vd s\r) (P_t|\nabla f|). 
\end{align*}

  Keeping the assumptions of Lemma \ref{Hess-prop-1}, we have
\begin{align*}
    |\Hess _{P_tf}|&\leq  \e^{-{Kt}/{2}}(P_t|\nabla f|^2)^{1/2}\l[\l(\int_0^t\e^{-Ks}\dot{k}(s)^2\,\vd s\r)^{1/2}+\alpha_1 \l(\int_0^t\e^{-Ks}k(s)^2\,\vd s\r)^{1/2}\r]\\
    &\quad+ \frac{\beta}{2}\e^{-{Kt}/{2}} (P_t|\nabla f|) \l(\int_0^t\e^{-{Ks}/{2}}k(s)\,\vd s\r).
\end{align*}
Choose  the  function
\begin{align*}
k(s):=\frac{\int_0^s\e^{Kr}\, \vd r}{\int_0^t\e^{Kr}\, \vd r}.
\end{align*}
Then we obtain
\begin{align*}
|\Hess _{P_tf}|\leq & \l(1+\frac{\alpha_1}{\sqrt{K}}\sqrt{\int_0^t\e^{Kr}\, \vd r}\r)\frac{\e^{-Kt/2}}{\sqrt{\int_0^t\e^{Kr}\, \vd r}}(P_t|\nabla f|^2)^{1/2}+ \frac{\beta }{K}\e^{-\frac{K}{2}t}(P_t|\nabla f|).
\end{align*}
\end{proof}

\begin{corollary}\label{esti-Hess}
  Assume that  $\alpha_1:=\|R\|_{\infty}<\infty$ (or $\alpha_2:=\|\tilde R\|_{\infty}<\infty$) and $\beta:=\|\nabla \Ric_V^{\sharp}+\vd^*R+R(\nabla V)\|_{\infty}<\infty$.  For any constant
  $K\in \R$,
  \begin{enumerate}[\rm(i)]
  \item if\/ $\Ric_V\geq K>0$, then for any $f\in C_b^2(M)$ and
    $t> 0$,
    \begin{align*}
      |\Hess_{P_tf}|_\HS^2\leq
    n \l(1+\l(\frac{\alpha_1}{\sqrt{K}}+\frac{\beta }{K}\r)\sqrt{\int_0^t\e^{Kr}\, \vd r}\r)^2\frac{\e^{-Kt}}{\int_0^t\e^{Kr}\, \vd r} P_t|\nabla f|^2;
    \end{align*}
  \item if\/ $\Ric_V=K>0$, 
 then for any $f\in C_b^2(M)$ and $t> 0$,
\begin{align*}
      |\Hess_{P_tf}|_\HS^2\leq
      \l(1+\l(\frac{\alpha_2}{\sqrt{K}}+\frac{\beta \, n}{K}\r)\sqrt{\int_0^t\e^{Kr}\, \vd r}\r)^2\frac{\e^{-Kt}}{\int_0^t\e^{Kr}\, \vd r} P_t|\nabla f|^2.
    \end{align*}
  \end{enumerate}
\end{corollary}

\begin{proof}
  These items are  direct consequences of Theorem
  \ref{general-hessian-theorem}. The second assertion
  can also be proved by an argument as in \cite[Theorem
  4.1]{Wang19} with some straightforward modifications. 
\end{proof}

In Theorem \ref{general-hessian-theorem2},  $|\nabla \Ric_V^{\sharp}+\vd^*R+R(\nabla V)|$ is assumed to be uniformly bounded on the whole space.  
We will relax this condition by regarding $|\nabla \Ric_V^{\sharp}+\vd^*R+R(\nabla V)|(x)$ as  a space dependent function with appropriate conditions. Let
\begin{align}
&\beta(x)=|\nabla \Ric_V^{\sharp}+\vd^*R+R(\nabla V)|(x);  \label{beta}\\
& K_V(x):=\inf\{\Ric_V(v,v)(x):\, v\in T_xM  \}. \label{KV}
\end{align}

\begin{theorem}\label{general-hessian-theorem3}
  Assume that there exist $K>0$, $p>1$ and $\delta>0$ such that   $K_V(x)- \frac{2(p-1)}{p}(\delta \beta(x))^{\frac{p}{p-1}}-K\geq 0$ for all $x\in M$.  Let $\alpha_1:=\|R\|_{\infty}<\infty$. Then for $f\in C_b^2(M)$,
 \begin{align*}
      |\Hess_{P_tf}|\leq
  \l(1+\frac{\alpha_1}{\sqrt{K}}\sqrt{\int_0^t\e^{Kr}\, \vd r}\r)\frac{\e^{-Kt/2}}{\sqrt{\int_0^t\e^{Kr}\, \vd r}}(P_t|\nabla f|^2)^{1/2}+  \frac{1}{\delta2^{(p-1)/p}(pK)^{1/p}} \e^{-\frac{K}{2}t} P_t|\nabla f|.
    \end{align*}
\end{theorem}

\begin{proof} 
It is easy to see from the condition that $K_V(x)\geq K>0$, i.e. $\Ric_V\geq K>0$. Following the steps of the proof of Theorem
\ref{general-hessian-theorem2}, it suffices to estimate 
$$\l|Q_t\int_0^tQ_r^{-1}\big(\nabla \Ric_V^{\sharp}+\vd^*R+R(\nabla V)\big)\big(Q_r(w), Q_r(k(r) w)\big)\,\vd r\r|.$$
For $p>1$,  by It\^{o}' s formula,
\begin{align*}
& \vd \l|Q_t\int_0^tQ_r^{-1}\big(\nabla \Ric_V^{\sharp}+\vd^*R+R(\nabla V)\big)\big(Q_r(w), Q_r(k(r) w)\big)\,\vd r\r|^{p} \\
&=- \frac{p}{2}\l|Q_t\int_0^tQ_r^{-1}\big(\nabla \Ric_V^{\sharp}+\vd^*R+R(\nabla V)\big)\big(Q_r(w), Q_r(k(r) w)\big)\,\vd r\r|^{p-2} \\
& \quad  \times  \Ric_V\bigg(Q_t\int_0^tQ_r^{-1}\big(\nabla \Ric_V^{\sharp}+\vd^*R+R(\nabla V)\big)\big(Q_r(w), Q_r(k(r) w)\big)\,\vd r, \\
&\qquad \qquad  Q_t\int_0^tQ_r^{-1}\big(\nabla \Ric_V^{\sharp}+\vd^*R+R(\nabla V)\big)\big(Q_r(w), Q_r(k(r) w)\big)\,\vd r\bigg)\, \vd t \\
&\quad + p\l|Q_t\int_0^tQ_r^{-1}\big(\nabla \Ric_V^{\sharp}+\vd^*R+R(\nabla V)\big)\big(Q_r(w), Q_r(k(r) w)\big)\,\vd r\r|^{p-2} \\
&\qquad \times \Big \langle \big(\nabla \Ric_V^{\sharp}+\vd^*R+R(\nabla V)\big)\big(Q_t(w), Q_t(k(t) w)\big), \\
&\qquad  \qquad Q_t\int_0^tQ_r^{-1}\big(\nabla \Ric_V^{\sharp}+\vd^*R+R(\nabla V)\big)\big(Q_r(w), Q_r(k(r) w)\big)\,\vd r \Big \rangle \, \vd t\\
&\leq -\frac{p}{2}K_V(X_t) \l|Q_t\int_0^tQ_r^{-1}\big(\nabla \Ric_V^{\sharp}+\vd^*R+R(\nabla V)\big)\big(Q_r(w), Q_r(k(r) w)\big)\,\vd r\r|^p\, \vd t\\
&\qquad + p\beta(X_t) |Q_t|^2k(t)\l|Q_t\int_0^tQ_r^{-1}\big(\nabla \Ric_V^{\sharp}+\vd^*R+R(\nabla V)\big)\big(Q_r(w), Q_r(k(r) w)\big)\,\vd r\r|^{p-1} \, \vd t.
\end{align*}
Using Young's inequality, we further obtain
\begin{align*}
& \vd \l|Q_t\int_0^tQ_r^{-1}\big(\nabla \Ric_V^{\sharp}+\vd^*R+R(\nabla V)\big)\big(Q_r(w), Q_r(k(r) w)\big)\,\vd r\r|^{p} \\
&\leq  \Big[ (p-1)(\delta \beta(X_t))^{\frac{p}{p-1}}-\frac{p}{2}K(X_t)\Big]\l|Q_t\int_0^tQ_r^{-1}\big(\nabla \Ric_V^{\sharp}+\vd^*R+R(\nabla V)\big)\big(Q_r(w), Q_r(k(r) w)\big)\,\vd r\r|^p\, \vd t \\
&\qquad \quad  +\frac{1}{\delta^{p}} |Q_t(w)|^{2p}\, \vd t \\
&\leq -\frac{p}{2}K\l|Q_t\int_0^tQ_r^{-1}\big(\nabla \Ric_V^{\sharp}+\vd^*R+R(\nabla V)\big)\big(Q_r(w), Q_r(k(r) w)\big)\,\vd r\r|^p\, \vd t +\frac{1}{\delta^{p}} |Q_t(w)|^{2p}\, \vd t,
\end{align*}
which further implies 
\begin{align*}
&\l|Q_{t\wedge \tau_D}\int_0^{t\wedge \tau_D}Q_r^{-1}\big(\nabla \Ric_V^{\sharp}+\vd^*R+R(\nabla V)\big)\big(Q_r(w), Q_r(k(r) w)\big)\,\vd r\r|\\
&\leq \frac{1}{\delta} \e^{-\frac{1}{2}K (t\wedge \tau_D)}\l(\int_0^{t\wedge \tau_D} \e^{\frac{p}{2}Ks}|Q_s(w)|^{2p}\, \vd s\r)^{1/p}\leq \frac{1}{\delta} \l( \frac{2}{ p K}\r)^{1/p}\e^{-\frac{1}{2}K(t\wedge \tau_D)},
\end{align*}
where $\tau_D$ is the first exit time of the compact set $D\subset M$.  Letting $D$ increase to $M$ yields
\begin{align*}
&\l|Q_{t}\int_0^{t}Q_r^{-1}\big(\nabla \Ric_V^{\sharp}+\vd^*R+R(\nabla V)\big)\big(Q_r(w), Q_r(k(r) w)\big)\,\vd r\r|  \leq \frac{1}{\delta} \l( \frac{2}{p K} \r)^{1/p} \e^{-\frac{1}{2}Kt}.
\end{align*}
\end{proof}


\section{The HSI inequality}

We first recall the formula relating relative entropy and Fisher information.
From now on, we always assume that $\nu$ is a distribution
which is absolutely continuous with respect to $\mu$ such that
$h:={\vd \nu}/{\vd \mu}\in C_b^2(M)$.

\begin{proposition}\label{auxi-prop}
  Assume that
$$\Ric_V:=\Ric-\Hess_V\geq K$$
for some positive constant $K$. Recall that
$\vd \nu^t=P_th\, \vd\mu$ for $t>0$. Then
\begin{enumerate}[\rm(i)]
\item (Integrated de Bruijn's formula)
  \begin{align*}
    H(\nu\,|\,\mu)=\Ent_{\mu}(h)=\frac12\int_0^{\infty}I_{\mu}(P_th)\, \vd t;
  \end{align*}
\item (Exponential decay of Fisher information) for every $t\geq0$,
  \begin{align*}
    I_{\mu}(P_th)=I(\nu^t\,|\, \mu)\leq \e^{-Kt}I(\nu\,|\, \mu)=\e^{-Kt}I_{\mu}(h).
  \end{align*}
\end{enumerate}
\end{proposition}

The HSI inequality connects the entropy $H$, the Stein
discrepancy $S$ and the Fisher information~$I$. We first give a bound
for the Fisher information by Stein's discrepancy $S$.
More precisely, we have the following result.

\begin{theorem}\label{esti-entropy}
Let $\nu$ be a distribution satisfying $\vd \nu=h\, \vd \mu$. 
 Assume that  $\alpha_1:=\|R\|_{\infty}<\infty$ (or $\alpha_2:=\|\tilde{R}\|_{\infty}<\infty$) and
$$\beta:=\|\nabla \Ric_V^{\sharp}+\vd^*R+R(\nabla V)\|_{\infty}<\infty.$$
\begin{enumerate}[\rm(i)]
  \item If\/ $\Ric_V\geq K,$ then for $t>0$ and $f\in C_b^2(M)$,
 \begin{align}\label{IS-formula-i}
  I_{\mu}(P_th)\leq \Psi(t)\, S^2(\nu\,|\, \mu),\quad  t>0,
\end{align}
where
$
\Psi(t)= \min\l\{\Psi_1(t),\, \Psi_2(t)\r\}
$ 
and
\begin{align*}
& \Psi_1(t):=  \frac{Kn}{\e^{2Kt}-\e^{Kt}} \l(1+\big(\frac{\alpha_1}{\sqrt{K}}+\frac{\beta }{K}\big)\l(\frac{\e^{Kt}-1}{K}\r)^{1/2}\r)^2;\\
& \Psi_2(t):= \frac{(K-2\alpha_1)n}{\e^{(2K-2\alpha_1)t}-\e^{Kt}} \l(1+\frac{\beta }{K}\l(\frac{\e^{Kt}-1}{K}\r)^{1/2}\r)^2.
\end{align*}
 \item 
If $\Ric_V=K>0$, then $\Psi$ in \eqref{IS-formula-i} also can be chosen as 
\begin{align}
&\min\l\{\tilde{\Psi}_1(t),\, \tilde{\Psi}_2(t)\r\},\ \    \mbox{where} \qquad \label{Psi-2}\\
& \tilde{\Psi}_1(t):=  \frac{K}{\e^{2Kt}-\e^{Kt}} \l(1+\big(\frac{\alpha_2}{\sqrt{K}}+\frac{\beta n }{K}\big)\l(\frac{\e^{Kt}-1}{K}\r)^{1/2}\r)^2;\notag\\
& \tilde{\Psi}_2(t):= \frac{K-2\alpha_2}{\e^{(2K-2\alpha_2)t}-\e^{Kt}} \l(1+\frac{\beta n }{K}\l(\frac{\e^{Kt}-1}{K}\r)^{1/2}\r)^2. \notag
\end{align}
  \end{enumerate}


\end{theorem}

\begin{proof}
  By Theorem \ref{general-hessian-theorem}, if $\Ric_V\geq K$,
 $\|R\|_{\infty}<\infty,$ and
  $\beta<\infty,$
  then
  \begin{align}\label{IS-ineq}
  |\Hess_{P_tf}|_\HS^2&\leq   \Psi(t)(P_t|\nabla f|^2). 
  \end{align}
 
  Let $g_t=\log P_th$.  By the symmetry of $(P_t)_{t\geq 0}$ in
  $L^2(\mu)$,
  \begin{align*}
    I_{\mu}(P_th)=-\int (Lg_t)P_th\, \vd\mu=-\int (LP_tg_t)h\, \vd\mu=-\int LP_tg_t\, \vd\nu.
  \end{align*}
  Hence, according to the definition of a Stein kernel, we have
  \begin{align*}
    I_{\mu}(P_th)
    &=-\int \langle \id, \Hess_{P_tg_t} \rangle_\HS\, \vd\nu -\int \langle \nabla V, \nabla P_tg_t \rangle\, \vd\nu\\
    &=\int \langle \tau_{\nu}-\id, \Hess_{P_tg_t} \rangle_\HS \, \vd\nu
  \end{align*}
  and hence by the Cauchy-Schwartz inequality,
  \begin{align*}
    I_{\mu}(P_th)&=\int \langle \tau_{\nu}-\id, \Hess_{P_tg_t} \rangle_\HS \, \vd\nu\\
                 &\leq \l(\int |\tau_{\nu}-\id|_\HS^2\, \vd\nu\r)^{1/2}\l(\int |\Hess_{P_tg_t}|^2_\HS\, \vd\nu\r)^{1/2}\\
                 &\leq \l(\int|\tau_{\nu}-\id|_\HS^2\, \vd\nu\r)^{1/2} \l(\Psi(t)\int P_t|\nabla g_t|^2\ \vd\nu\r)^{1/2},
  \end{align*}
  here we use \eqref{IS-ineq} by taking the function $g_t=\log P_th$
  inside.  Since
  \begin{align*}
    \int P_t|\nabla g_t|^2\, \vd \nu
    &=\int P_t|\nabla g_t|^2 h\,\vd\mu=\int |\nabla g_t|^2 P_th\, \vd\mu\\
    &=\int \frac{|\nabla P_th|^2}{P_th}\,\vd\mu= I_{\mu} (P_th),
  \end{align*}
  it then follows that
  \begin{align*}
    I_{\mu} (P_th)\leq \Psi(t) \int |\tau_{\nu}-\id|_\HS^2\,\vd\nu.
  \end{align*}
  Taking the infimum over all Stein kernels of $\nu$, we finish the
  proof of (i). The second item can be proved following the same steps as above by 
replacing the upper bound in \eqref{IS-ineq} by that in \eqref{esti-HS-ineq}.
\end{proof}

\begin{corollary}\label{esti-S}
  Assume that $\beta=0$ and $\|R\|_{\infty}<\infty$.  Let $\nu$ be a
  distribution satisfying $\vd \nu=h\, \vd \mu$ with $h\in
  C_b^2(M)$.
  \begin{enumerate}[\rm(i)]
  \item If\/ $\Ric_V\geq K,$ then for $t>0$,
    \begin{align}\label{IS-1}
      I_{\mu}(P_th)\leq \frac{n(2\|R\|_{\infty}-K)}{\e^{Kt}-\e^{2(K-\|R\|_{\infty})t}}\,S(\nu\,|\, \mu)^2.
    \end{align}
  \item If\/ $\Ric_V=K,$ then for $t>0$,
    \begin{align*}
      I_{\mu}(P_th)\leq \frac{2\|\tilde{R}\|_{\infty}-K}{\e^{Kt}-\e^{2(K-\|\tilde{R}\|_{\infty})t}}\,S(\nu\,|\, \mu)^2.
    \end{align*}
  \end{enumerate}
\end{corollary}
\begin{remark}\label{rem1}
When $M$ is a Ricci parallel manifold, $\nabla V$ is a
  Killing field, we have $\beta=0$ and in this case, we observe that when $\Ric_V=K>0$, both
  inequalities can be used to bound $I_{\mu}(P_th)$.  It is easy to see that
  when $K< 2(K-\|R\|_{\infty})$, the first inequality may give a
  smaller upper bound as the main decay rate is
  $\e^{-2(K-\|R\|_{\infty})t}$ which is faster than $\e^{-Kt}$. When
  $K< 2(K-\|\tilde{R}\|_{\infty})$ and if $\|\tilde{R}\|_{\infty}$ is small, then
  the second inequality is likely to give the sharper estimate as the upper bound in \eqref{IS-1}
  has an additional $n$.
\end{remark}

In case $|\nabla \Ric_V^{\sharp}+\vd^*R+R(\nabla V)|$ is not uniformly bounded, we have the following result.

\begin{theorem}\label{esti-S2}
Let $\nu$ be a
  distribution satisfying $\vd \nu=h\, \vd \mu$ with $h\in
  C_b^2(M)$. Assume that there exists $K>0$, $p>1$ and $\delta>0$ such that   $K_V(x)-\frac{2(p-1)}{p}(\delta \beta(x))^{\frac{p}{p-1}}\geq K$ for all $x\in M$, where $K_V$ and $\beta$ are defined as in \eqref{beta} and \eqref{KV}.  Moreover, assume that $\alpha_1:=\|R\|_{\infty}<\infty$. Then for $f\in C_b^2(M)$,
    \begin{align}\label{IS-1}
      I_{\mu}(P_th)\leq n \l(1+\l(\frac{\alpha_1}{\sqrt{K}}+\frac{1}{\delta2^{(p-1)/p}(pK)^{1/p}}\r)\sqrt{\int_0^t\e^{Kr}\,\vd r}\r)^2\frac{\e^{-Kt}}{\int_0^t\e^{Kr}\, \vd r} \,S(\nu\,|\, \mu)^2.
    \end{align}
\end{theorem}

Using Theorem \ref{esti-entropy}, we have the following inequality connecting the entropies $H,S$ and $I$.

\begin{theorem}[HSI inequality]\label{Gen-HSI-them}
  Suppose that $\Ric_V\geq K$ for some $K>0$. Let $\nu$ be a distribution satisfying
  $\vd \nu=h\,\vd \mu$. Assume that 
\begin{align*}
\|\Hess_{P_tf}\|_{\HS}^2\leq \Psi(t) P_t|\nabla f|^2,
\end{align*}
for some function $\Psi\in C([0,\infty))$.
then 
    \begin{align*}
  H(\nu\,|\, \mu)\leq \frac12\inf_{u>0}\l\{I(\nu\,|\,\mu)\int_0^u\e^{-Kt}\,\vd t+S(\nu\,|\,\mu)^2\int_u^{\infty}\Psi(t)\,\vd t\r\}.
\end{align*}

  \end{theorem}

\begin{proof}
By Proposition \ref{auxi-prop} (i), we have
 \begin{align*}
    H(\nu\,|\,\mu)=\frac12\int_0^{\infty}I_{\mu}(P_th)\, \vd t.
  \end{align*}
Combining this with the following facts:
 \begin{align*}
    I_{\mu}(P_th)\leq \e^{-Kt}I(\nu\,|\, \mu),
  \end{align*}
and 
\begin{align*}
I_{\mu}(P_th)\leq \Psi(t) S^2(\nu\,|\, \mu),
\end{align*}
we obtain
  \begin{align*}
    H(\nu\,|\, \mu)&\leq  \frac12\inf_{u>0}\Bigg\{I(\nu\,|\,\mu)\int_0^u\e^{-Kt}\,\vd t+S(\nu\,|\,\mu)^2\int_u^{\infty}\Psi(t)\,\vd t\Bigg\}.
  \end{align*}
\end{proof}

  \begin{remark}
 Suppose that $\beta=\|\nabla \Ric_V^{\sharp}+\vd^*R+R(\nabla V)\|_{\infty}<\infty$. If $\beta=0$, then
  combined with Corollary \ref{esti-S} (i), we get the HSI inequality
  \begin{align*}
    H(\nu\,|\, \mu)&\leq  \frac12\inf_{u>0}\Bigg\{I(\nu\,|\,\mu)\int_0^u\e^{-Kt}\,\vd t+nS(\nu\,|\,\mu)^2\int_u^{\infty}\l(\frac{K-2\|R\|_{\infty}}{\e^{(2K-2\|R\|_{\infty})t}-\e^{Kt}}\r)\,\vd t\Bigg\}.
  \end{align*}
  The term  $$\frac{K-2\|R\|_{\infty}}{\e^{(2K-2\|R\|_{\infty})t}-\e^{Kt}}$$
   has the decay rate at least $\e^{-Kt}$.
  If\/ $\beta\neq 0$, the decay rate  
$$n\l(\frac{K-2\|R\|_{\infty}}{\e^{(2K-2\|R\|_{\infty})t}-\e^{Kt}}\r)\l(1+\frac{\beta}{K}\l(\frac{\e^{Kt}-1}{K}\r)^{1/2}\r)^2$$
won't be faster than $\e^{-Kt}$.  In this case, using Corollary
\ref{esti-Hess}, the decay rate of
$$n \l(1+\l(\frac{\|R\|_{\infty}}{\sqrt{K}}+\frac{\beta }{K}\r)\l(\frac{\e^{Kt}-1}{K}\r)^{1/2}\r)^2\frac{\e^{-Kt}}{\int_0^t\e^{Kr}\, \vd r}$$
is the same as $\e^{-Kt}$. From this point of view, when
$\beta\neq 0$, we may choose the estimate from Corollary
\ref{esti-Hess} to establish the HSI inequality.

\end{remark}

To make the upper bounds in Theorem \ref {Gen-HSI-them} more explicit,
we continue the discussion by assuming
$|\nabla \Ric_V^{\sharp}+\vd^*R+R(\nabla V)|$ is bounded, dealing with
the cases $\beta=0$ and $\beta\neq 0$ separately. We also treat the case
that the norm is not bounded
but satisfies the specific conditions of Corollary~\ref{esti-S2}.

\subsection{Case I: $\beta=0$.}
We first introduce the main result of this subsection.

\begin{theorem}\label{case-0}
Assume that 
$\|R\|_{\infty}<\infty$ and $\beta=0$. 
\begin{enumerate}[\rm(i)]
  \item If\/ $\Ric_V\geq K>0$ and $\alpha:=K-2\|R\|_{\infty}>0$, then
    \begin{align}
      H(\nu\,|\, \mu)&\leq  \frac{I(\nu\,|\,\mu)}{2K}\l(1-\l(\frac{I(\nu\,|\,\mu)}{I(\nu\,|\,\mu)+\alpha nS^2(\nu\,|\,\mu)}\r)^{K/\alpha }\r)\notag\\
      &\quad+\frac{nS^2(\nu\,|\,\mu)}{2}\int_0^{\textstyle\frac{I(\nu\,|\,\mu)}{I(\nu\,|\,\mu)+\alpha nS^2(\nu\,|\,\mu)}}\frac{r^{K/\alpha }}{1-r}\,\vd r.\label{eq-H1}
    \end{align}
  \item [\rm(i')] If\/ $\Ric_V\geq K>0$ and
    $\alpha=K-2\|R\|_{\infty}=0$, then
    \begin{align*}
      H(\nu\,|\, \mu)\leq  \frac{I(\nu\,|\,\mu)}{2K}\l(1-\e^{-nKS^2(\nu\,|\,\mu)/I(\nu\,|\,\mu)}\r)+\frac{nS^2(\nu\,|\,\mu)}{2}\,{\rm li}(\e^{-nKS^2(\nu\,|\,\mu)/I(\nu\,|\,\mu)}),
    \end{align*}
    where ${\rm li}(x)=\int_0^x \frac{1}{\ln t}\,\vd t$ is the logarithmic
    integral function.
  \item If\/ $\Ric_V=K>0$ and $\tilde{\alpha}:=K-2\|\tilde{R}\|_{\infty}>0$, then
    \begin{align}\label{eq-H2}
      \hskip-.7cmH(\nu\,|\, \mu)\leq & \frac{I(\nu\,|\,\mu)}{2K}\l(1-\l(\frac{I(\nu\,|\,\mu)}{I(\nu\,|\,\mu)+\tilde{\alpha} S^2(\nu\,|\,\mu)}\r)^{K/\tilde{\alpha}}\r) \notag \\
& \quad +
                            \frac{S^2(\nu\,|\,\mu)}{2}\int_0^{\textstyle\frac{I(\nu\,|\,\mu)}{I(\nu\,|\,\mu)+\tilde{\alpha} S^2(\nu\,|\,\mu)}}\frac{r^{K/\tilde{\alpha}}}{1-r}\,\vd r.
    \end{align}
    Moreover, if  $\Hess_V=K$, then
    \begin{align*}
      H(\nu\,|\,\mu)\leq \frac12 S^2(\nu\,|\,\mu)\log \l(1+\frac{I(\nu\,|\,\mu)}{KS^2(\nu\,|\,\mu)}\r).
    \end{align*}
  \item [\rm(ii')] If\/ $\Ric_V\geq K>0$ and
    $\tilde{\alpha}=0$, then
    \begin{align*}
      H(\nu\,|\, \mu)\leq \frac{I(\nu\,|\,\mu)}{2K}\l(1-\e^{-KS^2(\nu\,|\,\mu)/I(\nu\,|\,\mu)}\r)+\frac{S^2(\nu\,|\,\mu)}{2}\,{\rm li}(\e^{-KS^2(\nu\,|\,\mu)/I(\nu\,|\,\mu)}),
    \end{align*}
    where ${\rm li}(x)=\int_0^x \frac{1}{\ln t}\,\vd t$ is again the logarithmic
    integral function.
  \end{enumerate}
\end{theorem}
\begin{proof}
  We only need to prove the first two estimates (i) and (i'); then
  (ii) and (ii') are obtained through replacing $nS^2(\nu\,|\,\mu)$ by
  $S^2(\nu\,|\,\mu)$,  and $\|R\|_{\infty}$ by
  $\|\tilde{R}\|_{\infty}$, respectively.    
  By Theorem \ref{Gen-HSI-them}~(i) and (ii), we have
  \begin{align*}
    H(\nu\,|\, \mu)
    &\leq \frac12\inf_{u>0}\l\{I(\nu\,|\,\mu)\int_0^u\e^{-Kt}\,\vd t+nS(\nu\,|\,\mu)^2\int_u^{\infty}\frac{\alpha}{\e^{Kt}(\e^{\alpha t}-1)}\,\vd t\r\}\\
    &=\frac12\inf_{u>0}\l\{\frac{I(\nu\,|\,\mu)(1-\e^{-Ku})}{K}+nS(\nu\,|\,\mu)^2\int_0^{\e^{-\alpha u}}\frac{r^{K/\alpha }}{1-r}\,\vd r\r\}.
  \end{align*}
In the sequel we write $I=I(\nu\,|\,\mu)$
and $S=S(\nu\,|\,\mu)$ for simplicity.
It is easy to see that $\inf$ is reached for
  $\e^{\alpha u}=(\alpha n  S^2+I)/I$ so that
  \begin{align}\label{alpha-large-0}
    H(\nu\,|\,\mu)\leq \frac{I}{2K}\l(1-\l(\frac{I}{I+\alpha n  S^2}\r)^{K/\alpha }\r)+\frac{nS^2}{2}\int_0^{\textstyle\frac{I}{I+\alpha n  S^2}}\frac{r^{K/\alpha }}{1-r}\,\vd r.
  \end{align}
  We thus obtain (i).  The case $\alpha=0$ can be dealt as limiting
  result of \eqref{alpha-large-0} when $\alpha$ tends to $0$, i.e.,
  \begin{align*}
    \lim_{\alpha\rightarrow 0}
    &\l\{\frac{I}{2K}\l(1-\l(\frac{I}{I+\alpha n S^2}\r)^{K/\alpha }\r)+\frac{nS^2}{2}\int_0^{\textstyle\frac{I}{I+\alpha n  S^2}}\frac{r^{K/\alpha }}{1-r}\,\vd r\r\}\\
    &=\frac{I}{2K}\l(1-\e^{-K n \frac{S^2}{I}}\r)+\frac{nS^2}{2}\lim_{\alpha\rightarrow 0}
      \int_0^{\textstyle\frac{I}{I+\alpha n S^2}}\frac{r^{K/\alpha }}{1-r}\,\vd r\\
    &=\frac{I}{2K}\l(1-\e^{-nK\frac{S^2}{I}}\r)+\frac{nS^2}{2}\lim_{\alpha\rightarrow 0}
      \int_{\textstyle\frac{n S^2}{I+\alpha n S^2}}^{1/\alpha}\frac{(1-\alpha t)^{K/\alpha }}{t}\,\vd t\\
    &=\frac{I}{2K}\l(1-\e^{-nK\frac{S^2}{I}}\r)+\frac{nS^2}{2}
      \int_{\textstyle\frac{n S^2}{I}}^{\infty}\frac{\e^{-Kt}}{t}\,\vd t\\
    &=\frac{I}{2K}\l(1-\e^{-nK\frac{S^2}{I}}\r)+\frac{nS^2}{2}\,{\rm li}(\e^{-K\frac{nS^2}{I}})
  \end{align*}
  which proves (i').
  
  If $\Hess_V=K$, by Obata's Rigidity Theorem
  (see \cite[Theorem 2]{Tashiro:1965} or \cite[Theorem 6.3]{WY2014}), if $\dim M\geq2$, then $M$ is
 isometric to $\mathbb{R}^n$ which implies $\Ric_V=K$, $\alpha_n=K$
  and $\beta=0$. Thus
  by \eqref{eq-H2},
  \begin{align*}
    H(\nu\,|\, \mu)
    &\leq  \frac{I(\nu\,|\,\mu)}{2K}\l(1-\l(\frac{I(\nu\,|\,\mu)}{I(\nu\,|\,\mu)+K S^2(\nu\,|\,\mu)}\r)\r)+
      \frac{S^2(\nu\,|\,\mu)}{2}\int_0^{\textstyle\frac{I(\nu\,|\,\mu)}{I(\nu\,|\,\mu)+K S^2(\nu\,|\,\mu)}}\frac{r}{1-r}\,\vd r\\
    &=\frac{S^2(\nu\,|\,\mu)I(\nu\,|\,\mu) }{2(I(\nu\,|\,\mu)+K S^2(\nu\,|\,\mu))}+
      \frac{S^2(\nu\,|\,\mu)}{2}\int_0^{\textstyle\frac{I(\nu\,|\,\mu)}{I(\nu\,|\,\mu)+K S^2(\nu\,|\,\mu)}}\l(\frac{1}{1-r}-1\r)\,\vd r\\
    &=\frac12 S^2(\nu\,|\,\mu)\log \l(1+\frac{I(\nu\,|\,\mu)}{KS^2(\nu\,|\,\mu)}\r),
  \end{align*}
  which covers the result in \cite[Theorem 2.2]{LNP15} for the Euclidean case $M=\R^n$
  and $\mu$ the standard Gaussian distribution on $\R^n$.
\end{proof}

\begin{remark}\label{rem3}
  In the case $\Ric_V=K>0$ and $\tilde{\alpha}>0$ (which implies $\alpha>0$), both 
  inequalities \eqref{eq-H1} and \eqref{eq-H2} hold. Hence one may choose
  the one which provides the sharper estimate.

\end{remark}

The case that $\beta=0$ and  $\alpha$ or $\tilde{\alpha}$ is less than 0,
can be dealt as follows.
\begin{theorem}\label{case-2}
  Assume that $\beta=0$ and $\|R\|_{\infty}<\infty$.
  \begin{enumerate}[\rm(i)]
  \item If\/ $\Ric_V\geq K>0$ and $\alpha:=K-2\|R\|_{\infty}<0$, then
    \begin{align}\label{HSI-2}
      H(\nu\,|\,\mu)\leq \frac{n S^2(\nu\,|\,\mu)\max\l\{-\alpha,K\r\}}{2K}\,\Theta\l(\frac{I(\nu\,|\,\mu)}{nS^2(\nu\,|\,\mu)\max\{-\alpha,K\}}\r)
    \end{align}
    where
    \begin{align*}
      \Theta(r)=\left\{
      \begin{array}{ll}
        1+\log r, &\ \  r\geq 1; \\
        r, &\ \  0<r<1.
      \end{array}
             \right.
    \end{align*}
  \item If\/ $\Ric_V= K>0$ and $\tilde{\alpha}:=K-2\|\tilde{R}\|_{\infty}<0$, then
    \begin{align*}
      H(\nu\,|\,\mu)\leq \frac{ S^2(\nu\,|\,\mu)\max\l\{-\tilde{\alpha},K\r\}}{2K}\,\Theta\l(\frac{I(\nu\,|\,\mu)}{S^2(\nu\,|\,\mu)\max\{-\tilde{\alpha},K\}}\r).
    \end{align*}
  \end{enumerate}

\end{theorem}
\begin{proof}
  As $\alpha:=K-2\|R\|_{\infty}<0$, we have
  \begin{align*}
    1-\e^{-K u}\leq \max\{1,-K/\alpha\}(1-\e^{\alpha u}),
  \end{align*}
  and then
  \begin{align*}
    \frac{-\alpha}{\e^{Ku}-\e^{(K+\alpha)u}}\leq \max\l\{-\alpha,K\r\}\frac{1}{\e^{Ku}-1},
  \end{align*}
  which implies
  \begin{align*}
    H(\nu\,|\,\mu)&\leq  I(\nu\,|\,\mu)\frac{1-\e^{-Ku}}{2K}+\frac{n}{2}S^2(\nu\,|\,\mu)\int_u^{\infty}\max\{-\alpha,K\}\,\frac{1}{\e^{Kt}-1}\,\vd t\\
                  &=I(\nu\,|\,\mu)\frac{1-\e^{-Ku}}{2K}-\frac{n}{2K}S^2(\nu\,|\,\mu)\max\l\{-\alpha, K\r\}\ln(1-\e^{-Ku}).
  \end{align*}
  This further implies
  \begin{align*}
    H(\nu\,|\,\mu)&\leq \frac12\inf_u\l\{I(\nu\,|\,\mu)\frac{1-\e^{-Ku}}{K}-\frac{nS^2(\nu\,|\,\mu)}{K}\max\l\{-\alpha, K\r\}\ln(1-\e^{-Ku})\r\}\\
                  &=\frac{n S^2(\nu\,|\,\mu)\max\l\{-\alpha,K\r\}}{2K}\,\Theta\l(\frac{I(\nu\,|\,\mu)}{nS^2(\nu\,|\,\mu)\max\{-\alpha,K\}}\r).\qedhere
  \end{align*}
\end{proof}

\subsection{Case II : $\beta\neq 0$.}

We start by introducing the main theorem of this subsection which also
provides a general way to the HSI inequality. 

\begin{theorem}\label{K-R1-s2}
Assume that $\alpha_1:=\|R\|_{\infty}<\infty, \  \beta:=\|\nabla \Ric_V^{\sharp}+\vd^*R+R(\nabla V)\|_{\infty}<\infty$.
Let $\vd \nu=h\, \vd \mu$ with $h\in C_0^{\infty}(M)$.
\begin{enumerate}[\rm(i)] 
  \item If $\Ric_V\geq K$, then
\begin{align*}
H(\nu\,|\,\mu)\leq \frac{n\l(1+\eps\r)S^2(\nu\,|\, \mu)}{2\eps}\l[ c_0 +\Theta\l(\frac{\eps I(\nu\,|\, \mu)}{n\l(1+\eps\r)KS^2(\nu\,|\, \mu)}-c_0\r)\r],
\end{align*}
for any $\eps>0$,  where $$c_0=\frac{\eps (\alpha_1 \sqrt{K}+\beta)^2}{K^3}-1.$$

Moreover, if $\alpha_1=0$ and $\beta=0$, then
\begin{align*}
H(\nu\,|\, \mu)\leq  \frac{n}{2} S^2(\nu\,|\,\mu)\ln\l(1+\frac{I }{nKS^2(\nu\,|\,\mu)}\r).
\end{align*}
 \item If $\Ric_V= K$, then
\begin{align*}
H(\nu\,|\,\mu)
\leq   \frac{\l(1+\eps\r)S^2(\nu\,|\, \mu)}{2\eps}\l[ \tilde{c}_0 +\Theta\l(\frac{\eps I(\nu\,|\, \mu)}{\l(1+\eps\r)KS^2(\nu\,|\, \mu)}- \tilde{c}_0\r)\r],
\end{align*}
for any $\eps>0$,  where $$ \tilde{c}_0=\frac{\eps (\alpha_2 \sqrt{K}+n\beta)^2}{K^3}-1.$$
 Moreover, if $\alpha_2=\beta=0$, then 
\begin{align*}
H(\nu\,|\, \mu)\leq  \frac{1}{2} S^2(\nu\,|\,\mu)\ln\l(1+\frac{I }{KS^2(\nu\,|\,\mu)}\r).
\end{align*}
\end{enumerate}
\end{theorem}

\begin{proof}
We only need to prove the first estimate. Denote again $I=I(\nu\,|\,\mu)$ and $S=S(\nu\,|\,\mu)$ for simplicity.
By  Theorem \ref{esti-entropy}, we have 
\begin{align*}
I_{\mu}(P_th)&\leq n\l(\frac{1}{\sqrt{\int_0^t\e^{Kr}\, \vd r}}+\frac{\alpha_1}{\sqrt{K}}+\frac{\beta}{K}\r)^2\e^{-Kt} S^2(\nu\,|\, \mu)\\
&\leq n\l(1+\frac{1}{\eps}\r)S^2(\nu\,|\, \mu)\frac{\e^{-Kt}}{\int_0^t \e^{Kr}\,\vd r}+n(1+\eps)\l(\frac{\alpha_1}{\sqrt{K}}+\frac{\beta}{K}\r)^2\e^{-Kt}S^2(\nu\,|\, \mu)
\end{align*}
for any $\eps>0$. 
Using this inequality, we need to estimate
\begin{align*}
H(\nu\,|\, \mu)&\leq \frac12 \inf_{u>0}\l\{A\int_0^u\e^{-Kt}\,\vd t+B\int_u^{\infty}\frac{K}{\e^{Kt}(\e^{K t}-1)}\,\vd t+C\int_u^{\infty}\e^{-Kt}\, \vd t\r\}\\
&=\frac12 \inf_{u>0}\l\{\frac{A(1-\e^{-Ku})+C\e^{-Ku}}{K}+B\int_0^{\e^{-K u}}\frac{r}{1-r}\,\vd r\r\},
\end{align*}
where 
\begin{align*}
&A=I(\nu\,|\, \mu);\ \  B=n\l(1+\frac{1}{\eps}\r)S^2(\nu\,|\, \mu);\\
&C=n(1+\eps)\l(\frac{\alpha_1}{\sqrt{K}}+\frac{\beta}{K}\r)^2S^2(\nu\,|\, \mu).
\end{align*}
It is easy to see that if $A\leq C$, then  $\inf$ is reached when $u$ tends to $\infty$; if 
$A>C$ however, then  $\inf$ is reached  for $\e^{K u}=\frac{A-C+BK}{A-C}$ so that
\begin{align*}
H(\nu\,|\,\mu)\leq \frac{C}{2K}+\frac{B}{2}\ln \l(1+\frac{A-C}{BK}\r).
\end{align*}
We then conclude that 
\begin{align}\label{HS-beta-0}
H(\nu\,|\,\mu)\leq \frac{B}{2}\l[ c_0 +\Phi\l(\frac{A}{BK}-c_0\r)\r],
\end{align}
where $$c_0=\frac{C-BK}{BK}=\frac{\eps (\alpha_1 \sqrt{K}+\beta)^2}{K^3}-1.$$

The proof of (ii) is the same by taking $B$ as
\begin{align*}
 &\l(1+\frac{1}{\eps}\r)S^2(\nu\,|\,\mu),
 \end{align*}
and  $C$ as
$$(1+\eps)\l(\frac{\alpha_2}{\sqrt{K}}+\frac{n\beta}{K}\r)^2S^2(\nu\,|\, \mu).$$
The details are omitted there.
\end{proof}

\subsection{Case III: $|\nabla \Ric_V^{\sharp}+\vd^*R+R(\nabla V)|$ is not bounded}
For the case that $|\nabla \Ric_V^{\sharp}+\vd^*R+R(\nabla V)|$ is not bounded on the whole space $M$, we get the following result from Theorem \ref{esti-S2}.
\begin{theorem}\label{K-R1-s3}
  Assume that there exists $K>0$, $p>1$ and $\delta>0$ such that
  $$K_V(x)-\frac{2(p-1)}{p}\big(\delta \beta(x)\big)^{\frac{p}{p-1}}-K\geq 0$$ for all $x\in M$.  Let $\alpha_1:=\|R\|_{\infty}<\infty$. Then for $f\in C_b^2(M)$,
   \begin{align*}
H(\nu\,|\,\mu)\leq \frac{n^2\l(1+\eps\r)S^2(\nu\,|\, \mu)}{2\eps}\l[ \tilde{c}_0 +\Theta\l(\frac{\eps I(\nu\,|\, \mu)}{n^2\l(1+\eps\r)KS^2(\nu\,|\, \mu)}-\tilde{c}_0\r)\r],
\end{align*}
for any $\eps>0$,  where
$$\tilde{c}_0=\frac{\eps}{K}\l(\frac{\alpha_1}{\sqrt{K}}+\frac{1}{\delta2^{(p-1)/p}(pK)^{1/p}}\r)^2-1.$$

\end{theorem}

\begin{proof}
By Theorem \ref{esti-S2}, taking 
\begin{align*}
&A=I(\nu\,|\, \mu);\ \  B=n\l(1+\frac{1}{\eps}\r)S^2(\nu\,|\, \mu);\\
&C=n(1+\eps)\Big(\frac{\alpha_1}{\sqrt{K}}+\frac{1}{\delta2^{(p-1)/p}(pK)^{1/p}}\Big)^2S^2(\nu\,|\,\mu)
\end{align*}
in inequality \eqref{HS-beta-0} completes the proof.
\end{proof}

\subsection{Examples}\label{sub-section-exmaples}
To elucidate the conditions in Theorem \ref{case-0} and Theorem
\ref{case-2} we consider some examples. For simplicity, we restrict ourselves
to the case $\beta=0$.
For the case $\beta>0$, one may work out specific examples by using
Theorem \ref{case-0} directly.

\begin{example}\label{ep1}
  Let $M=\R^n$. Consider the operator $L=\Delta-x\cdot \nabla$.  We
  have $\Ric_V=1$, $R=0$ and $\nabla V=x$. Then
  $\mu(\vd x)=(2\pi)^{-n/2}\e^{-|x|^2/2}\, \vd x$, and by
  Theorem \ref{case-0} (ii), we have
  \begin{align*}
    H(\nu\,|\,\mu)\leq \frac12 S^2(\nu\,|\,\mu)\log \l(1+\frac{I(\nu\,|\,\mu)}{S^2(\nu\,|\,\mu)}\r),
  \end{align*}
  which covers the result in \cite{LNP15}.
\end{example}

\begin{example}
Let $M=\R$.  We consider a family of diffusion operator on the line of the type
\begin{align*}
Lf=f''-u'f'
\end{align*}
associated to the symmetric invariant probability measure $\vd \mu=\e^{-u}\,\vd x$
where $u$ is a smooth potential on $\R$. 
We have $\Ric=0$ and $R=0$. Thus
$$\Ric_V=u'',\quad \nabla \Ric_V^{\sharp}+\vd^*R+R(\nabla V)=u'''$$
Hence, if there exists $K>0$, $p>1$ and  $\delta>0$  such that $u''-\frac{2(p-1)}{p}|\delta u'''|^{\frac{p}{p-1}}\geq K>0$, 
then, for any $\eps>0$,
 \begin{align*}
H(\nu\,|\,\mu)
   &\leq \frac{\l(1+\eps\r)S^2(\nu\,|\, \mu)}{2\eps}\\
   &\quad\times\l[\frac{\eps  }{\delta^2 2^{2(p-1)/p}(pK)^{2/p}K}-1 +\Theta\l(\frac{\eps I(\nu\,|\, \mu)}{\l(1+\eps\r)KS^2(\nu\,|\, \mu)}-\frac{\eps  }{\delta^2 2^{2(p-1)/p}(pK)^{2/p}K}+1 \r)\r].
\end{align*}

In particular, if $\eps=\delta^2 2^{2(p-1)/p}(pK)^{2/p}K$, then
\begin{align*}
H(\nu\,|\,\mu)\leq \frac{\l(1+\delta^2 2^{2(p-1)/p}(pK)^{2/p}K\r)S^2(\nu\,|\, \mu)}{\delta^2 2^{1+2(p-1)/p}(pK)^{2/p}K}\,\Theta\l(\frac{\delta^2 2^{2(p-1)/p}(pK)^{2/p} I(\nu\,|\, \mu)}{\l(1+\delta^2 2^{2(p-1)/p}(pK)^{2/p}K\r) S^2(\nu\,|\, \mu)}\r).
\end{align*}
For instance, let $u=\frac{1}{2}(x^2+ax^4)$ with $a>0$. Then
$u''=1+6a x^2$ and $u'''=12ax$. Note that $|u'''|$ is unbounded on $\R$.
Let $p=2$ and $\delta^2=\frac{1}{24a}$.  Then
\begin{align*}
u''-(\delta u''')^2\geq 1
\end{align*}  
and 
\begin{align*}
H(\nu\,|\,\mu)\leq \frac{1}{2}(1+6a)\,S^2(\nu\,|\, \mu) \,\Theta\l(\frac{I(\nu\,|\, \mu)}{\l(6a+1\r) S^2(\nu\,|\, \mu)}\r).
\end{align*}
Note that \cite[Proposition 4.5]{LNP15} requires the following conditions to be satisfied:
there exists a constant $c>0$ such that
\begin{align*}
&u''\geq c,\\
& u^{(4)}-u'u'''+2(u'')^2-6cu''\geq 0,\\
& 3(u''')^2\leq 2(u''-c)\l(u^{(4)}-u'u'''+2(u'')^2-6cu''\r).
\end{align*} 
Then it holds 
 \begin{align*}
 H(\nu\,|\,\mu)\leq \frac{1}{2}S^2(\nu\,|\, \mu)\,\Theta\l (\frac{I(\nu\,|\, \mu) }{c S^2(\nu\,|\, \mu)}\r).
 \end{align*}
 Obviously this result depends on properly choosing the constant $c$ and
 requires some computation compared to our conditions.

\end{example}


\begin{example}\label{ep3}
  Let $M=\mathbb{S}^n$.  Consider the operator $L=\Delta$ with
  $V\equiv 0$ and let $\mu(\vd x)= \vol(\vd x)/ \vol(M)$. Then
  $R_{ijk\ell}=(\delta_{ik}\delta_{j\ell}-\delta_{i\ell}\delta_{jk})$,
  $\Ric=n-1$, $\|\tilde{R}\|_{\infty}=\sqrt{2n(n-1)}$ and
  $$\alpha=K-2\|\tilde{R}\|_{\infty}=(n-1)-2\sqrt{2n(n-1)}<0.$$ By Theorem \ref{case-2}, we have
    \begin{align*}
      H(\nu\,|\,\mu)\leq \frac{\l(2\sqrt{2n(n-1)}-(n-1)\r)}{2(n-1)}\,S^2(\nu\,|\,\mu)\,\Theta\l(\frac{I(\nu\,|\,\mu)}{\big(2\sqrt{2n(n-1)}-(n-1)\big)S^2(\nu\,|\,\mu)}\r).
    \end{align*}
    On the other hand, to put these results in perspective with the
    method of \cite{LNP15}, let us first recall the necessary notions:
    \begin{align*}
      &\Gamma_1(f,\, g):=\langle \nabla f,\, \nabla g\rangle,\\
      &\Gamma_2(f,\, g):=\Ric_V(\nabla f,\, \nabla g)+\langle \Hess_f, \, \Hess_g \rangle_{\HS},\\
      &\Gamma_3(f,\, g):=\frac{1}{2}\Big(L\Gamma_2(f,\, g)-\Gamma_2(Lf,\, g)-\Gamma_2(f,\, Lg)\Big).
    \end{align*}
    Adopting the approach of \cite[Theorem 4.1]{LNP15} we have the
    following result.
    \begin{theorem}\label{Gamma-theorem}
      If there exist positive constants $\kappa, \rho$ and $\sigma$
      such that
      \begin{align*}
        &\Gamma_2(f)\geq \rho \Gamma_1(f),\quad \Gamma_3(f)\geq \kappa \Gamma_2(f),\quad \Gamma_2(f)\geq \sigma |\Hess_f|^2_{\HS}, 
      \end{align*}
      then
      \begin{align*}
        H(\nu\,|\,\mu)\leq \frac{1}{2\sigma }S^2(\nu\,|\,\mu)\,\Theta
        \l (\frac{\sigma\max\{\rho,\kappa\}I(\nu\,|\,\mu)}{\rho \kappa S^2(\nu\,|\,\mu)}\r).
      \end{align*}
    \end{theorem}
    For the general Riemmanian case, a crucial difficulty in applying Theorem \ref{Gamma-theorem} is to check the existence of a constant
    $\kappa>0$ such that $\Gamma_3(f)\geq\kappa\Gamma_2(f)$. In the
    special case $\mathbb{S}^n$, we have
    \begin{align*}
      \Gamma_2(f)&=(n-1)|\nabla f|^2+|\Hess_f|_{\HS}^2 \geq (n-1)|\nabla f|^2,\\
      \Gamma_3(f)&= (n-1)\l((n-1)|\nabla f|^2+|\Hess_f|_{\HS}^2\r)+\frac{1}{2}|\nabla \Hess_f|^2+2(n-1) |\Hess_f|_{\HS}^2
                   -2\langle \Hess_f (R^{\sharp, \sharp}), \Hess_f \rangle \\
                 &\geq \min\l\{(3(n-1)-2\|\tilde{R}\|_{\infty}), (n-1)\r\}\Gamma_2(f) \geq \l(3(n-1)-2\|\tilde{R}\|_{\infty}\r)\Gamma_2(f), \\
      \Gamma_2(f)&\geq |\Hess_f|_{\HS}^2.
    \end{align*}
    Thus $\rho=(n-1), \sigma=1, $ and
    $\kappa= \min\l\{(3(n-1)-2\|\tilde{R}\|_{\infty}), (n-1)\r\}$. If
    $\kappa=3(n-1)-2\sqrt{2n(n-1)}>0$, i.e. $n\geq 9$, by Theorem
    \ref{Gamma-theorem}, we have
    \begin{align*}
      H(\nu\,|\,\mu)\leq \frac{1}{2}S^2(\nu\,|\, \mu)\,\Theta\l (\frac{I(\nu\,|\,\mu)}{\l(3(n-1)-2\sqrt{2n(n-1)}\r)S^2(\nu\,|\, \mu)}\r).
    \end{align*}
    We first observe that this inequality holds for all $n\geq 0$ and when
$$I(\nu\,|\,\mu)\leq \l(3(n-1)-2\sqrt{2n(n-1)}\r)S^2(\nu\,|\,\mu);$$ 
the inequality can not become the classical log-Sobolev inequality.
In any case, our HSI inequality improves the classical log-Sobolev
inequality. In particular, for general Riemannian case, if $|R|$ is
small such that $K-2\|R\|_{\infty}>0$, the HSI inequality improves the
classical HI inequality no matter whether $S^2(\nu\,|\,\mu)$ is
small or not.

\end{example}


\begin{example}
  Let $G$ be a $n$-dimensional Lie group with a bi-invariant metric
  $g$, and let $\mathfrak{g}$ denote its Lie algebra.
  Consider $L=\Delta-\nabla V$
  for $V\in C^2(M)$ such that $\mu(\vd x)=\e^{-V(x)}\vd x$. Then for
  $X,Y,Z\in \mathfrak{g}$,
  \begin{align*}
    \nabla _X Y=\frac12 [X,Y] \quad \mbox{and}\ \quad R(X,Y)Z=\frac{1}{4}[Z,[X,Y]].
  \end{align*}
  By the Jacobi identity, we have
  \begin{align*}
    &\big(\nabla \Hess_V+R(\nabla V)\big)(X,Y)\\
    &\quad=\nabla_X(\nabla_Y\nabla V)-\nabla_{\nabla_X^Y}\nabla V+R(\nabla V,X)Y\\
    &\quad=\frac{1}{4}[X,[Y,\nabla V]]+\frac{1}{4}[\nabla V,[X,Y]]+\frac{1}{4}[Y,[\nabla V,X]]=0.
  \end{align*}
  We conclude that if $G$ is a Ricci parallel Lie group with
  $\Ric_V\geq K>0$ and $\|R\|_{\infty}<\infty$, then the inequalities
  in Theorem \ref{case-0}\,(i) and Theorem \ref{case-2}\,(i) hold. When the
  condition $\Ric_V=K>0$ is satisfied, both of the inequalities in Theorems
  \ref{case-0} and \ref{case-2} (i) and (ii) hold true.

\end{example}

\section{The WS inequality and HWSI inequality}\label{Section-WS-WSH}
Denote by $\sP(M)$ the set of probability measures on $M$.  For
$\mu_1,\mu_2\in\sP(M)$ the $L^2$-Wasserstein distance is given by
$$\W_2(\mu_1,\mu_2):=\inf_{\pi\in\sC(\mu_1,\mu_2)}
\left(\iint_{M\times M}\rho(x,y)^2\, \vd\pi( x, y)\right)^{1/2}$$
where $\rho$ denotes the Riemannian distance on $M$ and
$\sC(\mu_1,\mu_2)$ consists of all couplings of $\mu_1$ and~$\mu_2$.
The Wasserstein distance has various characterizations and plays an
important role in the study of SDEs, partial differential equations,
optimal transportation problems, etc.  For more background, one may
consult \cite{Talagrand, vonRS05,Wbook14} and the references therein.
The following Theorem describes the relationship between Wasserstein distance
and Stein discrepancy.

\begin{theorem}[WS inequality]\label{WS-ineq-RM}
   Assume that $\Ric_V\geq K>0$, $\alpha_1:=\|R\|_{\infty}<\infty$\,( or $\alpha_2:=\|\tilde{R}\|_{\infty}<\infty$) and
$$\beta:=\|\nabla \Ric_V^{\sharp}+\vd^*R+R(\nabla V)\|_{\infty}<\infty.$$
Then  for $\nu\in \mathscr{P}(M)$ satisfying
${\vd \nu}/{\vd \mu}\in C_b^2(M)$, we have
$$
\W_2(\nu, \mu)\leq \l(\int_0^{\infty}\sqrt{\Psi(t)}\, \vd t\r)
S(\nu\,|\,\mu),
$$
 where  $\Psi$ is defined by the term in \eqref{IS-1} (and also as in \eqref{Psi-2} when $\Ric=K>0$).

\end{theorem}

\begin{proof}
  Recall that $h={\vd\nu}/{\vd\mu}\in C_b^2(M)$ and let
  $\vd \nu^t=P_th\,\vd \mu$. 
  By the
  formula in \cite[Lemma 2]{OV00} or \cite[Theorem 24.2(iv)]
  {Villani}, we obtain
  \begin{align}\label{W-S} {\frac{\vd^+}{\vd t}}\W_2(\nu,\nu^t)\leq
    \l(\int_M\frac{|\nabla P_th|^2}{P_th}\, \vd\mu\r)^{1/2}=
    I_{\mu}(P_th)^{1/2},
  \end{align}
  where ${\frac{\vd^+}{\vd t}}$ stands for the upper right derivative.
  On the other hand, by Theorem \ref{esti-entropy},
  \begin{align*}
    I_{\mu} (P_th)\leq \Psi(t) S(\nu\,|\,\mu)^2.
  \end{align*}
  Combining this with \eqref{W-S}, we obtain
  \begin{equation*}
    \W_2(\nu, \mu)\leq \int_0^{\infty} (I_{\mu}(P_th))^{1/2}\, \vd t\leq S(\nu\,|\,\mu)\int_0^{\infty}\sqrt{\Psi(t)}\, \vd t.\qedhere
  \end{equation*}
\end{proof}

\begin{corollary}\label{Wasserstein-Stein}
 Assume that $\beta=0$ and $\|R\|_{\infty}<\infty$.  Let
  $\nu\in \mathscr{P}(M)$ satisfying
  ${\vd \nu}/{\vd \mu}\in C_b^2(M)$.
  \begin{enumerate}[\rm(i)]
  \item If\/ $\Ric_V\geq K>0,$ then
    \begin{align*}
      \W_2(\nu, \mu)\leq \l(\int_0^{\infty}\sqrt{\frac{n(2\|R\|_{\infty}-K)}{\e^{Kt}-\e^{2(K-\|R\|_{\infty})t}}}\, \vd t\r)\, S(\nu\,|\,\mu);
    \end{align*}
  \item if\/ $\Ric_V=K>0,$ then
    \begin{align*}
      \W_2(\nu, \mu)\leq \l(\int_0^{\infty}\sqrt{\frac{2\|\tilde{R}\|_{\infty}-K}{\e^{Kt}-\e^{2(K-\|\tilde{R}\|_{\infty})t}}}\, \vd t\r)\, S(\nu\,|\,\mu).
    \end{align*}
  \end{enumerate}
\end{corollary}

One may compare this inequality with the classical
Talagrand-type transportation cost inequality
\begin{align}\label{W-H}
  \W_2(\nu, \mu)^2\leq \frac{1}{2K}H(\nu\,|\,\mu).
\end{align}
We can go
further and improve this inequality to the following HWSI
inequality by assuming $\beta=0$.

\begin{theorem}[HWSI inequality]\label{WSH-ineq}
  Assume that 
$\|R\|_{\infty}<\infty$ and $\beta=0$. If\/ $\Ric_V\geq K>0$ and $\alpha:=K-2\|R\|_{\infty}>0$. Let
  $\vd\nu=h\,\vd \mu$. Then
  \begin{align*}
    \W_2(\nu , \,\mu)\leq \frac{S(\nu\,|\, \mu)}{2K}\int_0^{L^{-1}\l(\frac{2K H(\nu\,|\, \mu)}{S^2(\nu\,|\, \mu)}\r)} \frac{1}{\sqrt{y}}\l(1-\l(\frac{y}{y+\alpha n}\r)^{K/\alpha}\r)\, \vd y
  \end{align*}
  where
  \begin{align*}
L(x)=x+Kn \int_0^x\frac{r^{K/\alpha-1}(r-x)}{(r+\alpha n)^{K/\alpha +1}}\, \vd r.
\end{align*}
  
\end{theorem}

\begin{remark}\label{rem4}
  Since $L(r)\leq r$ for $r\geq 0$,
  this inequality improves the Talagrand quadratic
  transportation cost inequality \eqref{W-H}.
\end{remark}

\begin{proof}[Proof of Theorem \ref{WSH-ineq}]
Recall that $\vd \nu^t= P_th \,\vd \mu$. Then
\begin{align*}
H(\nu^t\,|\ \mu)=\frac{1}{2}\int_0^{\infty}I_{\mu}(P_{s+t}h)\, \vd s.
\end{align*}
Together with Proposition \ref{auxi-prop} this implies
 \begin{align*}
 H(\nu^t\,|\, \mu) &\leq \frac{1}{2}\inf_{u>0} \l\{ \{ I(\nu^t\,|\,\mu)\int_0^u \e^{-Ks}\, \vd s +  S(\nu\,|\,\mu)^2 \int_{u+t}^{\infty} \Psi(s)\, \vd s \r\} \\
 &\leq \frac{1}{2}\inf_{u>0} \l\{ \{ I(\nu^t\,|\,\mu)\int_0^u \e^{-Ks}\, \vd s +  S(\nu\,|\,\mu)^2 \int_{u}^{\infty} \Psi(s)\, \vd s \r\}
 \end{align*}
If $\beta=0$,  $\alpha=K-2\|R\|_{\infty}\geq 0$ and $\Psi(s)=\frac{\alpha n}{\e^{Ks}(\e^{\alpha s}-1)}$, then 
\begin{align}
      H(\nu^t\,|\, \mu)&\leq  \frac{I(\nu^t\,|\,\mu)}{2K}\l(1-\l(\frac{I(\nu^t\,|\,\mu)}{I(\nu^t\,|\,\mu)+\alpha nS^2(\nu\,|\,\mu)}\r)^{K/\alpha }\r)\notag\\
      &\quad+\frac{nS^2(\nu\,|\,\mu)}{2}\int_0^{\textstyle\frac{I(\nu^t\,|\,\mu)}{I(\nu^t\,|\,\mu)+\alpha nS^2(\nu\,|\,\mu)}}\frac{r^{K/\alpha }}{1-r}\,\vd r \notag \\
      &=\frac{S^2(\nu\,|\, \mu)}{2K} L\l(\frac{I(\nu^t\,|\, \mu)}{S^2(\nu\,|\, \mu)}\r),\label{eq-H1}
    \end{align}
where
\begin{align*}
L(x)=x+Kn \int_0^x\frac{r^{K/\alpha-1}(r-x)}{(r+\alpha n)^{K/\alpha +1}}\, \vd r.
\end{align*}
It is easy to see that 
$$L'(x)=1-\l(\frac{x}{x+\alpha n}\r)^{K/\alpha}>0$$ 
for $x>0$. Thus $L^{-1}$ exists and   
\begin{align*}
I(\nu^t\,|\, \mu)\geq  S^2(\nu\,|\, \mu) L^{-1}\l(\frac{2KH(\nu^t\,|\, \mu)}{S^2(\nu\,|\, \mu)}\r).
\end{align*}
Dividing $\W_2(\mu, \nu^t)$ by $t$ and using the above estimate, we have
\begin{align*}
\frac{\vd^+}{\vd t} \W_2(\mu, \nu^t)&\leq I_{\mu}(P_th)^{1/2}=\frac{\frac{\vd }{\vd t} H(\nu^t\,|\, \mu)}{ \sqrt{I(\nu^t\,|\, \mu)}} \\
& \leq  \frac{-\frac{\vd }{\vd t} H(\nu^t\,|\, \mu)}{S(\nu\,|\, \mu)\sqrt{ L^{-1}\l(\frac{2KH(\nu^t\,|\, \mu)}{S^2(\nu\,|\, \mu)}\r)}}.
\end{align*}
Therefore, integrating both sides from $0$ to $\infty$ yields
\begin{align*}
\W_2(\nu,\mu)&\leq \int_0^{\infty}\frac{-\frac{\vd }{\vd t} H(\nu^t\,|\, \mu)}{S(\nu\,|\, \mu)\sqrt{ L^{-1}\l(\frac{2K H(\nu^t\,|\, \mu)}{S^2(\nu\,|\, \mu)}\r)}}\\
&=S(\nu\,|\, \mu)\int_0^{\frac{H(\nu\,|\, \mu)}{S^2(\nu\,|\, \mu)}}\frac{\vd x}{\sqrt{L^{-1}(2Kx)}}\\
&=\frac{S(\nu\,|\, \mu)}{2K}\int_0^{L^{-1}\l(\frac{2K H(\nu\,|\, \mu)}{S^2(\nu\,|\, \mu)}\r)} \frac{1}{\sqrt{y}}\l(1-(\frac{y}{y+\alpha n})^{K/\alpha}\r)\, \vd y.
\end{align*}

\end{proof}
In particular, if $\Hess_V=K$  for some positive constant $K$,
then by Obata's Rigidity Theorem (see \cite[Theorems 3.4 and 6.3]{WY2014}), $M$ is
isometric to $\mathbb{R}^n$, and we have
\begin{corollary}\label{cor-flat-case}
Assume that  $\Hess_V=K>0$. Let $\vd \nu=h \vd \mu$. Then
 \begin{equation*}
    \W_2(\nu,\,\mu)\leq \frac{S(\nu\,|\,\mu)}{K^{1/2}}\arccos\l(\exp\l(-\frac{H(\nu\,|\,\mu)}{S^2(\nu\,|\,\mu)}\r)\r).
    \qedhere
  \end{equation*}
\end{corollary}

\begin{proof}
As $\Hess_V=K$, we know that $M$ is isometric to $\R^n$. 
  First, we repeat the steps of the proof of Theorem \ref{WSH-ineq} letting
  $\Psi(t)= \frac{K}{\e^{Kt}(\e^{Kt}-1)}$. 
  By this and \eqref{W-S}, we obtain
  \begin{align*}
    \frac{ \vd}{ \vd t}\W_2(\nu,\nu^t)
    &\leq \sqrt{I(\nu^t\,|\,\mu)}\leq
      -\frac{\frac{ \vd}{ \vd t}H(\nu^t\,|\, \mu)}{\sqrt{K}S(\nu\,|\,\mu)\sqrt{\exp\l(\frac{2H(\nu^t\,|\,\mu)}{S^2(\nu\,|\,\mu)}\r)-1}}\\
    &=-\frac{\vd}{ \vd t}\l\{\frac{S(\nu\,|\,\mu)}{K^{1/2}}
      \arccos\l(\exp\l(-\frac{H(\nu^t\,|\,\mu)}{S^2(\nu\,|\,\mu)}\r)\r)\r\}.
  \end{align*}
  Consequently,
  \begin{equation*}
    \W_2(\nu,\,\mu)=\int_0^{\infty}\frac{\vd }{\vd t} \W_2(\mu,\nu^t)\, \vd t\leq \frac{S(\nu\,|\,\mu)}{K^{1/2}}\arccos\l(\exp\l(-\frac{H(\nu\,|\,\mu)}{S^2(\nu\,|\,\mu)}\r)\r).
    \qedhere
  \end{equation*}
\end{proof}

\section{Moment bounds and Stein discrepancy}\label{Section-concentration}
In \cite{LNP15}, the authors investigate another feature of Stein's
discrepancy applied to concentration inequalities on $\R^d$. It is
well known that the classical log-Sobolev inequalities on the
manifolds is a powerful tool towards the invariant measure. In this
section, we continue to relate the Stein discrepancy to the concentration
inequality on a Riemannian manifold.  Let
\begin{align*}
  S_p(\nu\,|\,\mu)=\inf \l(\int |\tau_{\nu}-\id|_\HS^p\,\vd \nu\r)^{1/p}.
\end{align*}
As explained in \cite{LNP15}, the growth of the Stein
discrepancy $S_p(\nu\,|\,\mu)$ in $p$ entails concentration properties
of the measure $\nu$ in terms of the growth of its moments. The
following result shows how to directly transfer information on the
Stein kernel to concentration properties on the manifold.
\begin{theorem}[Moment bounds]
  Assume that $\Ric_V\geq K>0$, and
$$|\Hess_{P_tf}|_{\HS}^2\leq \Psi(t)P_t|\nabla f|^2$$
where $\Psi$ satisfies
\begin{align*}
  \int_0^{\infty}\Psi^{1/2}(r)\, \vd r<\infty.
\end{align*}
There exists a numerical constant $C>0$ such that for every
$1$-Lipshitz function $f\colon M\rightarrow \R$ with $\int f \, \vd \nu=0$,
and every $p\geq 2$,
\begin{align*}
  \l(\int |f|^p\, \vd \nu\r)^{1/p}\leq C\l(S_p(\nu\,|\,\mu)+\sqrt{p}\Big(\int |\tau_{\nu}|_{\rm op}^{p/2}\vd \nu\Big)^{1/p}\r),
\end{align*}
where the constant $C$ depends on the constants $K$, $p$ and
$\int_0^{\infty}\Psi^{1/2}(r)\,\vd r.$
\end{theorem}
\begin{proof}
  We only prove the result for $p$ an even integer, the general case
  follows similarly with some further technicalities. We may also
  replace the assumption $\int_Mf \, \vd \nu=0$ by
  $\int_Mf \, \vd \mu=0$ via a simple use of the triangle inequality.
  Let $f\colon M\rightarrow \R$ be 1-Lipshitz, and assume $f$ to be smooth and
  bounded. Let $q\geq 1$ be an integer and set
  \begin{align*}
    \phi(t)=\int_M (P_tf)^{2q}\, \vd\nu,\quad t\geq 0.
  \end{align*}
  Since $\mu(f)=0$, it follows that $\phi(\infty)=0$.  Now using the
  calculation with respect to the semigroup $P_t$, we have
  \begin{align}\label{esti-phi}
    \phi'(t)&=2q \int_M (P_tf)^{2q-1}LP_tf\, \vd\nu\notag\\
            &=2q\int(P_tf)^{2q-1}\Delta P_tf\, \vd \nu-\int\langle \tau_{\nu},\Hess((P_tf)^{2q}) \rangle_\HS\, \vd\nu\notag\\
            &=2q\int (P_tf)^{2q-1}\langle {\id}-\tau_{\nu},\Hess(P_tf) \rangle_\HS\, \vd\nu\notag\\
            &\quad -2q(2q-1)\int_M(P_tf)^{2q-2}\langle \tau_{\nu}, \nabla P_tf\otimes \nabla P_tf \rangle\, \vd\nu.
  \end{align}
  Next, Theorem \ref{general-hessian-theorem} implies
  \begin{align*}
    \langle \tau_{\nu}-{\id}, \Hess(P_tf) \rangle_\HS
    &\leq |\tau_{\nu}-{\id}|_\HS\,|\Hess(P_tf)\|_\HS\\
    & \leq |\tau_{\nu}-{\id}|_\HS\, \l(\Psi(t) P_t|\nabla f|^2\r)^{1/2}\\
    & \leq |\tau_{\nu}-{\id}|_\HS \,\Psi^{1/2}(t).
  \end{align*}
  Combining these inequalities with \eqref{esti-phi} and observing that
  \begin{align*}
    |\nabla P_tf|\leq \e^{-K/2}P_t|\nabla f|\leq \e^{-K/2},
  \end{align*}
  we arrive at
  \begin{align*}
    -\phi'(t)&\leq \Psi^{1/2}(t)\int 2q|P_tf|^{2q-1}|\tau_{\nu}-{\id}|_\HS\, \vd\nu\\
             &\quad+\e^{-Kt}\int 2q(2q-1)(P_tf)^{2q-2}|\tau_{\nu}|_{\rm op}\,\vd\nu.
  \end{align*}
  Therefore, from the Young-H\"{o}lder inequality, we obtain
  \begin{align*}
    -\phi'(t)\leq C(t)\phi(t)+D(t),
  \end{align*}
  where
  $$D(t)=\Psi^{1/2}(t)\int|\tau_{\nu}-{\id}|_\HS^{2q}\,
  \vd\nu+\e^{-Kt}\int\big((2q-1)|\tau_{\nu}|_{\rm op}\big)^q \vd\nu$$ and
$$C(t)=\Psi^{1/2}(t)\,(2q)^{2q/(2q-1)}+\e^{-Kt}(2q)^{2q/(2q-2)}.$$
Thus we get
\begin{align*}
  \phi(t)\leq \int_t^{\infty}\exp\left(\int_t^sC(r)\,dr\right)D(s)\, \vd s,
\end{align*}
and it follows that
\begin{align*}
  \phi(0)
  &\leq \frac{1}{(2q)^{2q/(2q-1)}}\exp{\l((2q)^{2q/(2q-1)}\int_0^{\infty}\Psi^{1/2}(s)\, \vd s\r)}\int|\tau_{\nu}-{\id}|_\HS^{2q}\, \vd\nu\\
  &\quad+\frac{\e^{(2q)^{2q/(2q-2)}/K}}{(2q)^{2q/(2q-2)}}\int\l((2q-1)|\tau_{\nu}|_{\rm op}\r)^q\, \vd\nu.
\end{align*}
Therefore, there exists a constant $C>0$ such that
\begin{equation*}
  \int_M|f|^{2q}\, \vd\nu
  \leq C\l(\int|\tau_{\nu}-{\id}|_\HS^{2q}\, \vd\nu+\int\big(2q|\tau_{\nu}|_{\rm op}\big)^q\, \vd\nu\r).\qedhere
\end{equation*}
\end{proof}

\begin{remark}\label{rem2}
  We see that when $\Hess_V=K$, by Obata's Rigidity Theorem
  (see \cite[Theorem 3.4]{WY2014}), $M$ is
  isometric to $\mathbb{R}^n$, which implies $\Ric_V=K$, $\alpha_n=K$, $\|R\|_\infty=0$,
  and then the constant $C$ is independent of the dimension
  $n$. In the general case however, as $\Psi$ depends on the
  dimension, the constant $C$ will not be dimension-free.

  When $p=2$, we observe that
  $|\tau_{\nu}|_{\rm op}\leq 1+|\tau_{\nu}-\id|_\HS$ which implies
  that
  \begin{align*} {\rm Var}_{\nu}(f)\leq
    C(1+S(\nu\,|\,\mu)+S^2(\nu\,|\,\mu)).
  \end{align*}
  Thus, the Stein discrepancy $S(\nu\,|\,\mu)$ with respect to the
  invariant measure gives another control of the spectral properties
  for log-concave measures, see \cite{Milman} for the Lipshitz
  characterization of Poincar\'{e} inequalities for measures of this
  type.
\end{remark}

\bibliographystyle{amsplain}%
\bibliography{Stein-kernel}
\end{document}